\documentclass[11pt]{article}
\usepackage{amsfonts,latexsym,amssymb,amsthm,amsmath,graphicx,cases,mathrsfs}
\usepackage[dvipsnames]{xcolor}
\usepackage[ulem=normalem,draft]{changes}
\usepackage{paralist}
\usepackage{graphics} 
\usepackage{epsfig}
\usepackage{epstopdf}
\usepackage{ulem}
\usepackage[titletoc,title]{appendix}
\usepackage{empheq}
\usepackage{cancel}
\usepackage{tikz}
\usetikzlibrary{arrows,shapes}
\usetikzlibrary{decorations.pathmorphing,decorations.pathreplacing}
\usetikzlibrary{calc,patterns,angles,quotes}

\usepackage{todonotes}

\usepackage[colorlinks=true]{hyperref}
\hypersetup{urlcolor=blue, citecolor=red}
\allowdisplaybreaks
\newtheorem{theorem}{Theorem}[section]

\newtheorem{assumption}{Assumption}
\newtheorem{consequence}{Consequence}
\newtheorem{thm}{Theorem}

\newtheorem{lemma}[theorem]{Lemma}
\newtheorem{prop}{Proposition}
\newtheorem{clr}[theorem]{Corollary}

\theoremstyle{definition}

\newtheorem{rmk}{Remark}

\newcommand{\be}{\begin{equation}}
\newcommand{\ee}{\end{equation}}
\newcommand{\bsubeq}{\begin{subequations}}
	\newcommand{\esubeq}{\end{subequations}}

\newcommand{\ds}{\displaystyle}

\newcommand{\calL}{{\mathcal{L}}}
\newcommand{\calH}{{\mathcal{H}}}

\newcommand{\calI}{{\mathcal{I}}}

\newcommand{\calF}{{\mathcal{F}}}

\newcommand{\calD}{{\mathcal{D}}}

\newcommand{\calA}{{\mathcal{A}}}

\newcommand{\calU}{{\mathcal{U}}}

\newcommand{\calY}{{\mathcal{Y}}}

\newcommand{\calT}{{\mathcal{T}}}
\newcommand{\calV}{{\mathcal{V}}}

\newcommand{\calW}{{\mathcal{W}}}

\newcommand{\BR}{\mathbb{R}}

\newcommand{\BN}{\mathbb{N}}

\newcommand{\al}{\alpha}

\newcommand{\wti}{\widetilde}

\newcommand{\mb}[1]{\mathbf{#1}}

\newcommand{\bpm}{\begin{pmatrix}}
	\newcommand{\epm}{\end{pmatrix}}

\newcommand{\bbm}{\begin{bmatrix}}
	\newcommand{\ebm}{\end{bmatrix}}

\numberwithin{equation}{section}
\numberwithin{thm}{section}
\numberwithin{rmk}{section}
\numberwithin{prop}{section}

\newcommand{\bs}[1]{\boldsymbol{#1}}

\newcommand\rfrac[2]{{}^{#1}\!/_{#2}}

\newcommand{\norm}[1]{\left\lVert#1\right\rVert}
\newcommand{\abs}[1]{\left\lvert#1\right\rvert}

\newcommand{\ip}[2]{\langle #1, #2 \rangle}
\newcommand{\ipp}[2]{( #1, #2 )}
\newcommand{\Uad}{U_{ad}}

\newcommand{\nin}{\noindent}
\newcommand{\V}{\mathcal{V}}
\newcommand{\calP}{\mathcal{P}}

\newcommand{\ul}[1]{\underline{#1}}
\newcommand{\ol}[1]{\overline{#1}}

\newcommand{\tcr}[1]{\textcolor{red}{#1}}

\usepackage{todonotes}

\begin{document}
	\title{Continuous Differentiability of the Value Function of Semilinear Parabolic Infinite Time Horizon Optimal Control Problems on $L^2(\Omega)$ under Control Constraints \thanks{The authors were supported by the ERC advanced grant 668998 (OCLOC) under the EU's H2020 research program.}}
	
\author{Karl Kunisch \thanks{Institute for Mathematics and Scientific Computing, University of Graz, Heinrichstrasse 36, A-8010 Graz, Austria, and Radon Institute, Austrian Academy of Science, Linz, Austria. (karl.kunisch@uni-graz.at).}
    \and Buddhika Priyasad \thanks{Institute for Mathematics and Scientific Computing, University of Graz, Heinrichstrasse 36, A-8010 Graz, Austria. (b.sembukutti-liyanage@uni-graz.at).}}

	\date{\today}
	\maketitle
	
	\begin{abstract}
		\nin An abstract framework guaranteeing the local continuous differentiability of the value function associated with  optimal stabilization problems subject to  abstract semilinear parabolic equations  subject to a norm  constraint on the controls is established. It guarantees that the value function satisfies the associated Hamilton-Jacobi-Bellman equation in the classical sense. The applicability of the developed framework is demonstrated for specific semilinear parabolic equations.
		
	\end{abstract}
	
	\section{Introduction.}
	
	Continuous differentiability of the value function with respect to the initial datum is an important problem in optimal feedback control theory. Indeed, if the value function is $C^1$ then  it is the  solution of a Hamilton Jacobi Bellman (HJB) equation and its negative  gradient can be used to define on  optimal state feedback law. The subject matter of this paper addresses local continuous differentiability of the value function $\V$ for infinite horizon optimal control problems subject to semilinear parabolic control problems and norm constraints on the control. Such problems are intimately related to stabilization problems which are often cast as infinite horizon optimal control problems. Investigating infinite horizon problems constitutes one of the specificities of this paper. Another one is the fact that we focus on the differentiability of  $\V$ on (subsets of) $L^2(\Omega)$. Thus we need to consider the semilinear equations with initial data
	$y_0 \in L^2(\Omega)$. As a consequence the solutions of the semilinear equations only enjoy low Sobolev-space regularity. This restricts the class of nonlinearities, compared to those which are admissible if the states are in $L^\infty((0,\infty)\times\Omega)$, which is the situation typically addressed in the literature on optimal control \cite{Cas:1997} and \cite{Tro:2010}. The latter necessitates to take the initial conditions in  spaces strictly smaller than $L^2(\Omega)$. Here we consider $L^2(\Omega)$, first due to intrinsic interest, secondly because ultimately  the HJB equation should be solved numerically, which is easier in an $L^2(\Omega)$ setting than in other topologies, like $H^1(\Omega)$. Let us also recall that one of the approaches to solve the HJB equation is given by the policy iteration. It assumes that  the value function is $C^1$.
	
	The underlying analysis demands stability and  sensitivity analysis of infinite dimensional optimal control problems subject to nonlinear equations. For this purpose we utilize the theory of generalized equations as established by \cite{D:1995} and \cite{SMR:1980}. It involves first order approximations of the state and adjoint  equations, which lead to restrictions on the class of nonlinearities which can be admitted.  We refer to the section on examples in this respect.
	
	The current investigations are to some degree a continuation of work the first author's work on optimal feedback control for infinite dimensional systems. In  \cite{BKP:2018, BKP:TE:2018, BKP:2019} Taylor approximations of the value function for problems with a concrete structure, namely, bilinear control systems, and the Navier Stokes equations were investigated and differentiability of the value function was obtained as a by product. In these investigations norm constraints were not considered. Here we admit norm constraints and we focus  on semilinear equations. Let us also notice that the systems investigated in \cite{BKP:2018, BKP:TE:2018, BKP:2019} share the property  that the second derivatives with respect to the state variable of the nonlinearity  in the state equation do not depend on the state itself anymore.
	
	Let us also compare our work to the developments in the field of parametric sensitivity analysis of semilinear parabolic equations under control constraints.  There are many papers  focusing  on stability and sensitivity analysis of finite time horizon problems along with pointwise control constraints, see  e.g. \cite{BM:1999, GHH:2005, Gri:2004, GV:2006,  Mal:2002, MT:1999, Tro:2000, Wac:2005}, and the literature there. First, of these papers, except for  \cite{GV:2006, Tro:2000}, consider the case with initial data in $H^1(\Omega)$ or $C(\bar \Omega)$. In \cite{GV:2006} again  the third derivative of the nonlinearity is zero. Secondly, all of them consider the finite horizon  case. Since we treat  infinite horizon problems we have to  guarantee stabilizability (for small initial data) under control constraints. Then we use a fixed point argument to obtain well-posedness of the system.	Well-posedness  and stability with respect to parameters of the adjoint equation is significantly more involved for infinite horizon problems than for finite horizon problems. It requires techniques, differently from those used in the finite horizon case. Another aspect is the proper characterization of the adjoint state at $t=\infty$.	
	
	In the finite dimensional case, there is, of course a tremendous amount of work on the treatment of the value function if it is not $C^1$.  Fewer papers concentrate on the case where the value function enjoys smoothness properties. We mention \cite{Goe:2005} and \cite{CF:2013} in this respect.	
	
	In order to achieve the goal we desire, we lay out the following setup. In Section 2, we consider an abstract parametric optimization problem with an equality constraint and  another convex  constraint. Existence of an optimal solution, of a multiplier associate  to the equality constraint, and  Lipschitz stability of the component of the state variable which lies in the complement of the kernel of the linearized constraint will be established.
This result is necessary but not sufficient for the further developments, since stability is obtained in a norm which is too weak and since the stability estimate does not involve the component in the kernel of the linearized constraint and the multiplier, i.e. the adjoint states, yet. At the level of Section 2 this remains as Assumption \eqref{H4}. In Section 3 we specify the concrete optimal stabilization problem and a set of conditions, most importantly  on the nonlinearity of the state equation, under which Assumption \eqref{H4} can be established, for initial data $y_0 \in L^2(\Omega)$. Section 3 also contains a summary of the main results of this paper. They are stated as theorems with a little stronger assumptions than eventually necessary, for the saker of easing the presentation. Section 4 is dedicated to the proof of verifying the  assumptions of the general setup of Section 2 for the concrete optimal control problem stated in Section 3. As  conclusion we obtain the Lipschitz continuity in the appropriate norms of the all variables appearing in the optimality system with respect to the parameter of interest, which is the initial condition $y_0$, in our case. Since our analysis is a local one involving second order optimality conditions, solutions to the optimality system are related to local solutions to the optimal control problem. As a corollary to these results we obtain that the local value function is Fr\'{e}chet differentiable. In Section 5, we show that in the neighborhood of global solutions  the value function $\V$ satisfies the  Hamilton-Jacobi-Bellman (HJB) equation in the strong sense. Finally, Section 6 is devoted to demonstrating that the developed framework is applicable for some concrete examples, namely for linear systems, Fisher's equations, and  parabolic equations with global Lipschitz nonlinearities. All our results require a smallness assumption on the initial conditions $y_0$.
	Two aspects need to be taken into consideration in this respect. First $y_0$ has to be sufficiently small so that the controlled system is stable. Secondly a second order optimality condition is needed. For this to hold a sufficient condition is provided by smallness of the adjoint state, which in turn can be implied by smallness of $y_0$. We stress that these two issues are of related, but independent nature.

	\section{Lipschitz stability for an abstract optimization problem.}
	\label{Sec-abs_lip}	
	
	Here we present a stability result for an abstract, infinite dimensional optimization problem which will be the building block for the results below. This result is geared towards exploiting  the specific nature of optimization problem with differential equations as constraints. First  existence of a dual variable will result from a regular point condition. Subsequently the Lipschitz stability result is obtained in two steps. In the first one, we rely on the relationship between the linearized optimality conditions and an associated linear-quadratic optimal optimization problem, with an extra convex constraint. This approach is useful since it provides the existence of solutions to the linearized system on the basis of variational techniques. However it  dictates a certain norms for the involved quantities. These  norms are too weak for our goal of obtaining Lipschitz continuity of the adjoint variables in such a manner that differentiability of the cost with respect to the initial conditions can be argued. Therefore, in a second step we exploit the specific structure of the optimality systems, using the fact that it is related to a parabolic optimal control problem, to obtain the Lipschitz continuity in the stronger norms. This two step approach is also present in some of the earlier work on stability and sensitivity analysis which was quoted in the introduction. But due to that fact these papers considered finite horizon problems it came as a byproduct which improved the regularity of the adjoints. In our work it is essential to reach our  goal. This is why we decided to formalize this two step approach which was not done in earlier work.\\
	
	\nin Concretely, we consider the optimization problem	
	\begin{equation}\label{Pq}
	\tag{$P_q$}
	\begin{cases}
	\min \ f(x),\\
	e(x,q) = 0, \ x \in C.
	\end{cases}
	\end{equation}
	\nin with a parameter dependent equality constraint, and a general constraint described by $x \in C$, where $C$ is a closed convex subset of a real Hilbert space $X$. Further, $W$ is  a real Hilbert space and $P$ is a normed linear space. In the application that we have in mind, the parameter $q$ will appear as the initial condition in the dynamical system. The following Assumption \eqref{H0} is assumed to hold throughout. 	
	\begin{assumption}\label{H0}\ \\
		$q_0 \in P$ is a nominal reference parameter,\\
		$x_0$ is a local solution (\hyperref[Pq]{$P_{q_0}$}),\\
		$f: X \longrightarrow \BR^+$ is twice continuously differentiable in a neighborhood of $x_0$,\\
		\nin $e: X \times P \longrightarrow W$ is continuous, and twice continuously differentiable w.r.t. $x$, with first and second derivative  Lipschitz continuous in a neighborhood of $(x_0,q_0)$.
	\end{assumption}	
	\nin The derivatives {with respect to $x$} will be denoted  by primes and the derivatives w.r.t. $y$ and $u$ later on, are denoted by subscripts.
They are all considered in the sense of Lebesgue derivatives. \\
	
	\nin We introduce the Lagrangian $\calL: X \times P \times W^* \longrightarrow \BR$ associated to \eqref{Pq} by
	\begin{equation}\label{La_fu}
	\calL(x, q, \lambda) = f(x) + \ip{\lambda}{e(x,q)}_{W^*, W}.
	\end{equation}
	\nin Next further relevant assumptions are introduced:
	\begin{assumption}[regular point condition]\label{H1}
		\begin{equation*}
		0 \in \text{int } e'(x_0, q_0)(C - x_0),
		\end{equation*}
	\end{assumption}
	\nin where $int$ denotes the interior in the $W$ topology. This regularity condition implies the existence of a Lagrange multiplier $\lambda_0 \in W^*$, see e.g. \cite{MZ:1979} such that the following first order condition holds:
	\begin{equation}
	\begin{cases}
	\ip{\calL'(x_0, q_0, \lambda_0)}{c - x_0}_{X^*,X} \geq 0, \ \forall c \in C,\\
	e(x_0, q_0) = 0.
	\end{cases}
	\end{equation}
	\nin It is equivalent to
	\begin{equation}\label{eq:rp_h1}
	\begin{aligned}
	\begin{cases}
	0 \in \calL'(x_0, q_0, \lambda_0) + \partial \mb{I}_C(x_0), \quad &\text{in } X^*,\\
	e(x_0, q_0) = 0, \quad &\text{in } W,
	\end{cases}
	\end{aligned}	
	\end{equation}
	
\nin where $\partial \mb{I}_C(x)$ denotes the subdifferential of the indicator function of the set $C$ at $x \in X$. 	
	
	\nin Let $\ds A \in \calL(X, X^*)$ denote the operator representation of
$\ds \calL''(x_0, q_0, \lambda_0)$, i.e.
	\begin{equation}\label{def_A}
	\ip{Ax_1}{x_2}_{X^*,X} = \calL''(x_0, q_0, \lambda_0)(x_1 ,x_2)
	\end{equation}
	\nin and define
	\begin{equation}\label{def_E}
	E = e'(x_0, q_0) \in \calL(X,W).
	\end{equation}	
	\nin We further require
	\begin{assumption}[positive definiteness]\label{H2}
		\begin{equation*}
		\exists \kappa >  0: \ \ip{Ax}{x}_{X^*, X} \geq \kappa \norm{x}^2_X, \ \forall x \in \text{ker } E.
		\end{equation*}
	\end{assumption}	

Condition \eqref{H2} is a bit stronger than a second order sufficient optimality condition, since it does not take into consideration the activity or inactivity of the constraints. Such weaker second order conditions typically allow to derive quadratic positive definite  lower bounds on the cost and H\"older continuity with respect to perturbations. For Lipschitz continuity and differentiability stronger assumptions, such as \eqref{H2} are typically assumed. We refer exemplarily to  \cite{Gri:2004,GHH:2005, GV:2006, Wac:2005}, and \cite[Section 2.3]{IK:2008}. The constraints in these references, however, are not identical with those of the present paper.

The stability result of $(x_0, \lambda_0)$ with respect to perturbation of $q$ at $q_0$  will be based on Robinson's strong regularity condition which involves the following linearized form of the optimality condition,	
	\begin{equation}\label{str_reg_con}
	\begin{aligned}
	\begin{cases}
	0 \in \calL'(x_0, q_0, \lambda_0) + A(x - x_0) + E^*(\lambda - \lambda_0) + \partial \mb{I}_C(x)  &\text{in } X^*,\\
	0 = e(x_0, q_0) + E(x - x_0) &\text{in } W.
	\end{cases}
	\end{aligned}	
	\end{equation}
	\nin We define a multivalued operator $\ds \calT: X \times W^*  \longrightarrow X^* \times W$ by
	\begin{equation}\label{mult_map}
	\calT \bpm x \\ \lambda \epm = \bpm A & E^* \\ E & 0 \epm \bpm x \\ \lambda \epm + \bpm f'(x_0) - A x_0 \\ -E x_0 \epm + \bpm \partial \mb{I}_C(x) \\ 0 \epm,
	\end{equation}
	\nin and observe that \eqref{str_reg_con} is equivalent to \begin{equation*}
	\ds 0 \in \calT \bpm x \\ \lambda \epm.
	\end{equation*}	
	\nin Here it is understood that $\calT$ is evaluated at $(x_0,q_0,\lambda_0)\in X\times P\times W^*$. But $\calT$ is not yet the mapping for which we need to verify the Robinson-Dontchev strong regularity condition in our context. It relates to the fact that we must to treat the  multiplier $\lambda$ in smaller space than $W^*$. Before we can properly specify this condition some additional preparation is necessary. We first  introduce Banach spaces:
	\begin{equation}\label{eq:uline}
		\ul{X} \subset X, \ \ul{W^*} \subset W^*, \ \ul{X^*} \subset X^*,
	\end{equation}
	\nin with continuous injections. We emphasize  that $\ul{X^*}$ should not be confused  with $(\ul{X})^*$. A restriction of $\cal T$ will be defined as multivalued operator $\ds \ul{\calT}: \ul{X} \times \ul{W^*}  \to \ul{X^*} \times W$. Indeed, in applications to optimal control problems extra regularity of multipliers  can be obtained by investigating the solutions  \eqref{eq:rp_h1}, see e.g. Section \ref{sec3}. In the context of optimal stabilization problems this structural property will become transparent in Proposition \ref{prop:adj} and Proposition \ref{prop:est_p}, see also \cite[Proposition 15]{BKP:2019}. It will turn out to be essential for our purposes. But this situation where the multiplier has extra regularity   is also of abstract interest. When studying stability in this setting this means that the second coordinate of the domain of $\calT$ needs to be changed from $W^*$ to  $\ul{W^*}$. This entails that the range space of  $\calT$ has to be modified appropriately, in order to  obtain stability of the $\lambda$ coordinate. For this purpose we introduce $\ul{X^*} \subset X^*$. The reason for further restricting $X$ to $\ul{X}$ will become evident in the proof of Proposition \ref{prop:est_p}. It is related to the fact that we consider infinite horizon problems. A concrete use of these space is elaborated in detailed in subsection \ref{Sec-abs_su}.\\

	\nin Now we adapt the conditions on $f$ and $e$ to the choice of the spaces in \eqref{eq:uline}.

\begin{assumption}\label{H00} \ \\
		There exists a neighborhood $\ds \wti{U}_1 \times \wti{U}_2 \subset \ul{X} \times P$ of $(x_0, q_0)$ such that
		\begin{enumerate}[(i)]
			\item the restriction of $x \mapsto f'(x)$ to $\ul{X}$ defines a mapping $\ul{f'}(x)$ from $\wti{U}_1 \subset \ul{X}$ to $\ul{X^*}$,
			\item the restriction $e'(x,q)^* \in \calL(W^*,X^*)$ to $\ul{W^*}$ defines operators $\ul{e'(x,q)^*} \in \calL(\ul{W^*},\ul{X^*})$ for every $(x,q) \in \wti{U}_1 \times \wti{U}_2$.
		\end{enumerate}
	\end{assumption}
	\nin With these assumption holding  we define the restricted linearized Lagrangian
	\begin{equation}
		\ul{\calL'}: \wti{U}_1 \times \wti{U}_2 \times \ul{W^*} \subset \ul{X} \times P \times \ul{W^*} \longrightarrow \ul{X^*} \quad \text{by} \quad \ul{\calL'}(x,q,\lambda) = \ul{f'}(x) + \ul{e'(x,q)^*}\lambda.
	\end{equation}
	\nin Next we adapt $\partial \mb{I}_{C} \subset X^*$  to the situation of \eqref{eq:uline} and define for $x \in \ul{X}$ the set valued mapping
	\begin{equation}
		{\ul{\partial \mb{I}_{C}}(x)} = \left\{ y \in \ul{X^*}: \ip{y}{v - x}_{X^*,X} \leq 0, \ \forall v \in C \cap \ul{X} \right\} \subset  \ul{X^*}.
	\end{equation}
 	\nin We henceforth assume that $(x_0,\lambda_0) \in \ul{X} \times \ul{W^*}$, it will also follow as a special case of \eqref{H4} below.  The following assumption will guarantee that the restriction $\ul{\calT}$ of $\calT$ is well-defined as operator from  $\ul{X} \times \ul{W^*}$  to $\ul{X^*} \times W$, and the one beyond is needed for Lipschitz continuous dependence of local solutions to \eqref{Pq} with respect to $q$.
 		
	\begin{assumption}\label{H000}\ \\
		$\ds \ul{\calL'}: \wti{U}_1 \times \wti{U}_2 \times \ul{W^*} \subset \ul{X} \times P \times \ul{W^*} \longrightarrow \ul{X^*}$ is Fr\'{e}chet differentiable with respect to $x$, and $(\ul{\calL'})'$, as a mapping $(x,q,\lambda) \mapsto (\ul{\calL'})'(x,q,\lambda)$, is continuous at $(x_0,q_0,\lambda_0) \in \ul{X} \times P \times \ul{W^*}$.
	\end{assumption}
	
	\begin{assumption}\label{H3}
		There exists $\nu > 0$ such that:
		\begin{subequations}
			\begin{align}
			\norm{e(x,q_1) - e(x,q_2)}_W &\leq \nu \norm{q_1 - q_2}_{P}, \ \forall (x,q_1) \text{ and } (x,q_2) \in \wti{U}_1 \times \wti{U}_2, \label{lip_e}\\
			\norm{ {\ul{e'(x,q_1)^*} - \ul{e'(x,q_2)^*}} }_{\calL(\ul{W^*},\ul{X^*})} &\leq \nu \norm{q_1 - q_2}_{P}, \ \forall (x,q_1) \text{ and } (x,q_2) \in \wti{U}_1 \times \wti{U}_2. \label{lip_de}
			\end{align}
		\end{subequations}
	\end{assumption}
	\nin Let us further set
	\begin{equation*}
	\ul{E^*} = \ul{e'(x_0,q_0)^*} \in \calL(\ul{W^*},\ul{X^*}) \text{ and } \ul{A} = (\ul{\calL'}(x_0,q_0,\lambda_0))' \in \calL(\ul{X},\ul{X^*}).
	\end{equation*}
\nin With Assumptions \eqref{H0}-\eqref{H000} holding   \eqref{eq:rp_h1} can be expressed as
	\begin{equation}
	\begin{aligned}
	\begin{cases}
	0 \in \ul{\calL'}(x_0, q_0, \lambda_0) + \ul{\partial \mb{I}_C}(x_0), \quad &\text{in } \ul{X^*},\\
	e(x_0, q_0) = 0, \quad &\text{in } W.
	\end{cases}
	\end{aligned}	
	\end{equation}
Moreover  \eqref{str_reg_con} restricted to $\ul{X}\times \ul{X^*}$ result in:
\begin{equation}\label{str_reg_conbar}
	\begin{aligned}
	0 \in \begin{cases}
	\ul{\calL'}(x_0, q_0, \lambda_0) + \ul{A}(x - x_0) + \ul{E^*}(\lambda - \lambda_0) + \ul{\partial \mb{I}_C}(x) & \text{ in } \ul{X^*},\\
	e(x_0, q_0) + E(x - x_0) & \text{ in } W,
	\end{cases}
	\end{aligned}
	\end{equation}
	\nin and  the multivalued operator $\ds \ul{\calT}: \ul{X} \times \ul{W^*}  \longrightarrow \ul{X^*} \times W$  related  to \eqref{mult_map} is defined as
	\begin{equation}\label{mult_map_b}
	\ul{\calT} \bpm x \\ \lambda \epm = \bpm \ul{A} & \ul{E^*} \\ E & 0 \epm \bpm x \\ \lambda \epm + \bpm \ul{f'}(x_0) - \ul{A} x_0 \\ -E x_0 \epm
+ \bpm \ul{\partial \mb{I}_C}(x) \\ 0 \epm.
	\end{equation}
	\nin Observe that \eqref{str_reg_conbar} is equivalent to \begin{equation*}
	\ds 0 \in \ul{\calT} \bpm x \\ \lambda \epm.
	\end{equation*}

	\nin Existence and Lipschitz continuity of solutions in a neighborhood $(x_0, q_0, \lambda_0)$ will follow from the strong regularity assumption which requires us to show that there exist neighborhoods $\hat{V} \subset \ul{X^*} \times W$ of $0$ and $\hat{U} = \hat{U}_1 \times \hat{U}_2 \subset \ul{X} \times \ul{W^*}$ of $(x_0, q_0)$ such that  $\ul{\calT}^{-1}$ has the properties that $ \ul{\calT}^{-1}(\hat{V}) \cap \hat{U}$ is single-valued and that it is Lipschitz continuous from $\hat{V}$ to $\hat{U}$, see \cite{D:1995}, (and also  \cite{SMR:1980}, \cite[Definition 2.2, p 31]{IK:2008}, in case $\ul{X} =X, \ \ul{W^*} = W^*, \ \ul{X^*} =X^*$). We approach the strong regularity assumption in two steps. In the first one we argue invertibility of $\calT$ and Lipschitz continuity of the  variable $x$ in $X$. For this purpose we exploit  the symmetry of $\calT$ and consider an associated variational problem.  In our specific situation the inverse of $\calT$ - and consequently  of $\ul \calT$ - is single-valued and thus   the restriction to the neighborhood $\hat U$ is  not needed. Existence and Lipschitz continuity of $\lambda$ as well as Lipschitz continuity of $x$ in the small space $\ul{X}\times\ul{W^*}$ remains an assumption in the generality of problem \eqref{Pq}. It  will be  verified in a second step for the optimal stabilization problems in the following sections.

\begin{assumption}\label{H4}
		For  $(\beta_1, \beta_2) \in \hat{V} \subset \ul{X^*} \times W$ , the solution $\ds \left(x_{(\beta_1,\beta_2)}, \lambda_{(\beta_1,\beta_2)} \right)$ to $\ds  \calT \bpm x \\ \lambda \epm = \bpm \beta_1 \\ \beta_2 \epm$ lies in $\ul{X} \times \ul{W^*}$. Moreover there exists a constant $k > 0$ such that
		\begin{equation*}
		\norm{x_{(\beta_1,\beta_2)} - x_{(\hat{\beta}_1,\hat{\beta}_2)}}_{\ul{X}} + \norm{\lambda_{(\beta_1,\beta_2)} - \lambda_{(\hat{\beta}_1,\hat{\beta}_2)}}_{\ul{W^*}} \leq k \left[ \norm{(\beta_1,\beta_2) - (\hat{\beta}_1,\hat{\beta}_2)}_{\ul{X^*} \times W} + \norm{x_{(\beta_1,\beta_2)} - x_{(\hat{\beta}_1,\hat{\beta}_2)}}_X \right]
		\end{equation*}
		\nin for all $(\beta_1,\beta_2) \in \hat{V}, (\hat{\beta}_1,\hat{\beta}_2) \in \hat{V}$.
	\end{assumption}
	\nin This condition is used after the existence of $x_\beta=x_{(\beta_1,\beta_2)}$ was already established. Note that for $(\beta_1, \beta_2)^T = 0$ we have $(x_{(0,0)},\lambda_{(0,0)}) = (x_0, \lambda_0)$ and hence \eqref{H4} in particular implies that $(x_0, \lambda_0) \in \ul{X} \times \ul{W^*}$.
	\nin We arrive at the announced stability result.
	
	\begin{thm}\label{thm:lip_con}
		Assume that \eqref{H0}-\eqref{H4} hold at a local solution $x_0$ of (\hyperref[Pq]{$P_{q_0}$}). Then there exist a neighborhood $U = U(x_0,\lambda_0) \subset \ul{X} \times \ul{W^*}$, a neighborhood $N = N(q_0)\subset P$, and a constant $\mu$ such that for all $q \in N$ there exists a unique $(x(q), \lambda(q)) \in U$ satisfying
		\begin{equation}\label{eq:thm_lc_1} 
		\begin{aligned}
		0 \in \begin{cases}
		\ul{\calL'}(x(q), q, \lambda(q)) + \ul{\partial \mb{I}_C} (x(q)), \quad &\text{ in } \ul{X^*}, \\
		e(x(q), q), \quad &\text{ in } W,
		\end{cases}
		\end{aligned}		
		\end{equation}
		and
		\begin{equation}
		\norm{(x(q_1), \lambda(q_1)) - (x(q_2), \lambda(q_2))}_{\ul{X} \times \ul{W^*}} \leq \mu \norm{q_1 - q_2}_{P}, \ \forall q_1,q_2 \in N.\label{eq:thm_lc_2}
		\end{equation}
		 In addition there exists a nontrivial neighborhood $\wti{N} \subset N$ of $q_0$ such that $x(q)$ is a local solution of \eqref{Pq} for $q \in \wti{N}$.
	\end{thm}
	
	\nin For the proof we shall employ the following lemma in which $A \in \calL(X,X^*)$ and $E \in \calL(X,W)$ denote generic operators. For the sake of completeness we also include its proof.
	
	\begin{lemma}\label{lem:aux_opt}
		Let $(\tilde{a}, \tilde{b}) \in X^* \times W$, assume that $A\in \calL(X,X^*)$ is self-adjoint and  satisfies  \eqref{H2}, and that the set $\ds
S(\tilde{b}) = \{ x \in C: \ E x = \tilde{b} \}$ is nonempty. Then the problem
		\begin{equation}\label{eq:aux_opt_1}
		\begin{cases}
		\min_{x \in C} \tilde{J}(x) = \min_{x \in C} \frac{1}{2}\ip{Ax}{x}_{X^*,X} + \ip{\tilde{a}}{x}_{X^*,X},\\
		E x = \tilde{b},
		\end{cases}
		\end{equation}
		\nin admits a unique solution $x = x(\tilde{a}, \tilde{b})$ satisfying
		\begin{equation}
		\ip{A x + \tilde{a}}{v - x}_{X^*,X} \geq 0, \ \text{for all } v \in S(\tilde{b}).
		\end{equation}
		\nin If moreover the regular point condition $0 \in \text{int } E(C - x(\tilde{a}, \tilde{b}))$ holds, then there exists $\lambda = \lambda(\tilde{a}, \tilde{b}) \in W^*$ such that
		\begin{equation}\label{eq:aux_opt_rp}
		0 \in \begin{cases}
		\bpm A & E^* \\ E & 0\epm \bpm x \\ \lambda \epm + \bpm \tilde{a} \\ -\tilde{b} \epm + \bpm \partial \mb{I}_C(x) \\ 0 \epm.
		\end{cases}
		\end{equation}
	\end{lemma}
	\begin{proof}
		\nin Since $C$ is a closed and convex, $S(\tilde{b})$ is closed and convex. By assumption $S(\tilde{b})$ is nonempty. Hence there exists an $x \in C$ such that $Ex = \tilde{b}$.  Note that each such $x$ can be uniquely decomposed as  $x = w + y$, with  $y \in \text{ker} E,\, w \in \text{ker}E^\perp $ and $E w = \tilde{b}$.
		By \eqref{H2} the functional $\tilde{J}$ is bounded from below and coercive on $S(\tilde{b})$. Hence there exists a minimizing sequence $\{ x_n \}$ in $S(\tilde{b})$ such that $\ds \lim_{n \rightarrow \infty} \tilde{J}(x_n) = \inf_{x \in S(\tilde{b})} \tilde{J}(x)$. Each $x_n$ can be decomposed as  $x_n = w + y_n$, with  $y_n \in \text{ker} E$. By  \eqref{H2} the sequences $\ds \{y_n\}_{n=1}^\infty$ and hence  $\ds \{x_n\}_{n=1}^\infty$ are bounded. Thus there exists a subsequence  $\ds \{ x_{n_k} \}$ with weak limit $x=x(\tilde{a}, \tilde{b})$ in $S(\tilde{b})$. Since $\tilde{J}$ weakly lower semi-continuous, we have that $\ds \tilde{J}(x) \leq \liminf_{k \rightarrow \infty} \tilde{J}(x_{n_k})$ and $x$ minimizes $\tilde{J}$ over $S(\tilde{b})$. This further implies that $\ds \ip{Ax + \tilde{a}}{v -  x}_{X^*,X} \geq 0$ for all $v \in S(\tilde{b})$. Uniqueness of $x$ follows from \eqref{H2}.
		
		\nin The regular point condition implies the existence of a multiplier $\lambda = \lambda(\tilde{a}, \tilde{b})\in W^*$ such that \eqref{eq:aux_opt_rp} holds. See e.g. \cite[Theorem 1.6]{IK:2008}
	\end{proof}
	
	\bigskip
	
	\nin \textit{Proof of the Theorem \ref{thm:lip_con}.}
	\begin{enumerate}[(i)]	
		\item
		The proof of the first assertion of the Theorem \ref{thm:lip_con} is based on the implicit function theorem of Dontchev for generalized equations, see \cite[Theorem 2.4, Remark 2.5]{D:1995}. We introduce the mapping $\ul{\calF}: \ul{X} \times P \times \ul{W^*} \longrightarrow \ul{X^*} \times W$ given by
		\begin{equation*}
		\ul{\calF}(x,q,\lambda)= \bpm
		\ul{\calL'}(x,q,\lambda)  \\
		e(x,q) \epm,
		\end{equation*}
		and observe that   Assumption \eqref{H3} implies that for all $(x,q_1,\lambda),$ and $(x,q_2,\lambda) \in \wti{U}_1 \times \wti{U}_2 \times \ul{W^*}$
		\begin{equation}\label{eq:aux10}
			\norm{\ul{\calF}(x,q_1,\lambda) - \ul{\calF}(x,q_2,\lambda)}_{\ul{X^*} \times W} \le \nu \left( 1 + \norm{\lambda}_{\ul{W^*}} \right) \norm{q_1 - q_2}_{P}.
		\end{equation}
	By \eqref{H0} and \eqref{H000}, and using the integral mean value theorem it can be argued that
	\begin{equation*}
	\bpm x \\ \lambda \epm  \to  \bpm \ul{A} & \ul{E^*} \\ E & 0 \epm \bpm x \\ \lambda \epm + \bpm \ul{f'}(x_0) - \ul{A} x_0 \\ -E x_0 \epm
	\end{equation*}
	strongly approximates $\ul{\calF}$ at $(x_0,q_0,\lambda_0)$, in the sense of Dontchev,   \cite{D:1995}.  In the next two steps  the strong regularity condition for $\ul{\calT}$ will be  verified.		
		
		\item (Existence). \label{Thm-2.1.i} Let, at first, $(\beta_1, \beta_2) \in X^* \times W$ and consider $\ds \calT \bpm x \\ \lambda \epm = \bpm \beta_1 \\ \beta_2 \epm$ which is equivalent to,
		\begin{equation}\label{eq:lin_rp}
		0 \in \bpm a \\ -b \epm + \bpm A & E^* \\ E & 0 \epm \bpm x \\ \lambda \epm + \bpm \partial \mb{I}_C(x) \\ 0 \epm,
		\end{equation}
		\nin with $a = f'(x_0) - A x_0 - \beta_1, \ b = E x_0 + \beta_2$, and $A,E$ defined in \eqref{def_A}, \eqref{def_E}.
		\nin  To solve \eqref{eq:lin_rp} we consider,
		\begin{equation}\label{eq:aux_opt_2}
		\begin{cases}
		\min_{x \in C} \ \frac{1}{2}\ip{Ax}{x}_{X^*,X} + \ip{a}{x}_{X^*,X},\\
		E x = b.
		\end{cases}
		\end{equation}
		\nin This corresponds to \eqref{eq:aux_opt_1} with $\tilde{a} = a$, $\tilde{b} = b$ and feasible set $\ds S(\beta_2) = \{ x \in C: E x = b \}$. Clearly $\ds x_0 \in S(0) = \{ x \in C: E x = E x_0 \}$. By \eqref{H1} and \cite[Theorem I.2.8]{IK:2008}, there exists a neighborhood of the origin  $\ds \tilde{V} \subset X^* \times W$ such that $S(\beta_2)$ is not empty for all $(\beta_1, \beta_2) \in \tilde{V}$. Thus by Lemma \ref{lem:aux_opt} there exists a unique solution $x = x(\beta_1, \beta_2)$ to \eqref{eq:aux_opt_2} for each $(\beta_1, \beta_2) \in \tilde{V}$. By \cite[Theorem I.2.11, I.2.12, I.2.15]{IK:2008}, possibly after reducing $\tilde{V}$, these solutions depend H\"{o}lder continuously on $(\beta_1, \beta_2) \in \tilde{V} \subset X^* \times W$, with exponent $\frac{1}{2}$. The regular point condition for the solution $x(\beta_1, \beta_2)$ is
		\begin{equation*}
		0 \in \text{int } E (C - x(\beta_1, \beta_2)) = \text{int } E (C - x_0) - \beta_2,
		\end{equation*}
		\nin which is satisfied due to \eqref{H1}, possibly after again shrinking $\tilde{V}$. Hence there exists a Lagrange multiplier $\lambda = \lambda(\beta_1, \beta_2)$ associated to $E x = b$, and \eqref{eq:lin_rp} admits a unique solution $(x(\beta_1, \beta_2), \lambda(\beta_1, \beta_2))$ since it is the first order optimality condition for \eqref{eq:aux_opt_1}.
		
		\item (Uniqueness and Lipschitz continuity) Let $(\beta_1, \beta_2) \in \tilde{V}$ and $(\hat{\beta_1}, \hat{\beta_2}) \in \tilde{V}$ with corresponding solutions $(x, \lambda) \in X\times W^*$ and $(\hat{x}, \hat{\lambda}) \in X\times W^*$. This implies that
		\begin{equation}\label{eq:e_decom-1}
		\begin{cases}
		\ip{a + Ax + E^*\lambda}{c - x}_{X^*,X} \geq 0, \ \forall c \in C,\\
		Ex = b, \text{ with } a = f'(x_0) - Ax_0 - \beta_1, \ b = Ex_0 + \beta_2,
		\end{cases}
		\end{equation}
		and
		\begin{equation}\label{eq:e_decom-2}
		\begin{cases}
		\ip{\hat{a} + A\hat{x} + E^*\hat{\lambda}}{c - \hat{x}}_{X^*,X} \geq 0,\ \forall c \in C,\\
		E \hat{x} = \hat{b}, \text{with } \hat{a} = f'(x_0) - Ax_0 - \hat{\beta_1}, \ \hat{b} = Ex_0 + \hat{\beta_2}.
		\end{cases}
		\end{equation}
		By the first equations in \eqref{eq:e_decom-1} and \eqref{eq:e_decom-2} we obtain that
		\begin{equation}
		\ip{a + A x + E^* \lambda }{\hat{x} - x}_{X^*,X} \geq 0, \ \ip{\hat{a} + A \hat{x} + E^* \hat{\lambda} }{x - \hat{x}}_{X^*,X} \geq 0, \ x, \hat{x} \in C.
		\end{equation}
		\nin Combining  these inequalities, we have that
		\begin{equation}\label{ineq:dif_aAE}
		\ip{a - \hat{a} + A (x - \hat{x}) + E^* (\lambda - \hat{\lambda}) }{x - \hat{x}}_{X^*,X} \leq 0.
		\end{equation}
		\nin The second equalities in \eqref{eq:e_decom-1} and \eqref{eq:e_decom-2} imply that
		\begin{equation}\label{eq:E_dif}
		E(x - \hat{x}) = b - \hat{b}.
		\end{equation}
		\nin Let us set
		\begin{equation*}
		\delta x = \hat{x} - x, \ \delta \lambda = \hat{\lambda} - \lambda, \ \delta a = \hat{a} - a, \ \delta b = \hat{b} - b.
		\end{equation*}	
		Then $\delta \beta_1 = - (\hat{\beta_1} - \beta_1)$ and $\delta \beta_2 = \hat{\beta_2} - \beta_2$, and \eqref{ineq:dif_aAE}, \eqref{eq:E_dif} result in
		\begin{equation}\label{ineq:del_xAE}
		\ip{\delta x}{A \delta x}_{X,X^*} + \ip{\delta \lambda}{E \delta x}_{W^*,W} - \ip{\delta \beta_1}{\delta x}_{X^*,X} \leq 0,
		\end{equation}
		\nin and
		\begin{equation}\label{eq:del_E}
		E \delta x = \delta b.
		\end{equation}
		By \eqref{H1} the operator $E$ is surjective. Hence by the closed range theorem expressing $\delta x = \delta v + \delta w \in \text{ker } E + \text{range } E^*  $ implies that $\ds E \delta x = E \delta w = \delta \beta_2$. Again by the closed range theorem there exists $k_1 > 0$:
		\begin{equation}\label{eq:w_beta}
		\norm{\delta w}_X \leq k_1 \norm{\delta \beta_2}_{W}.
		\end{equation}
	From the first equation in \eqref{eq:lin_rp} we have
		\begin{equation*}
		Ax + E^*\lambda -Ax_0 +f'(x_0) - \beta_1 \in - \partial {\mb{I}}_C(x).
		\end{equation*}
	Next we restrict the perturbation parameters to satisfy $(\beta_1,\beta_2) \in (\underline{X^*}\times W) \cap \tilde V$. Due to Assumptions \eqref{H00} and \eqref{H4} we have $(x, \lambda) \in \ul{X} \times \ul{W^*}$,
		\begin{equation*}
			\ul{A} x + {\ul{E^*}} \lambda - \ul{A} x_0 + \ul{f'}(x_0) - \beta_1 \in \ul{X^*}
		\end{equation*}		
		and hence
		\begin{equation*}
			\ul{A} x + {\ul{E^*}} \lambda - \ul{A} x_0 + \ul{f'}(x_0) - \beta_1 \in - \ul{\partial \mb{I}_C}(x).
		\end{equation*}
		The analogous equation holds with $(x,\lambda, \beta_1)$ replaced by $(\hat x, \hat \lambda, \hat \beta_1)$.\\
By \eqref{H2}, \eqref{ineq:del_xAE} and Assumption \eqref{H4}  we find
		\begin{align}
		\kappa \norm{\delta v}^2 & \leq \ip{\delta v}{A \delta v}_{X, X^*} = \ip{\delta x}{A \delta x}_{X, X^*} - 2 \ip{\delta v}{A \delta w}_{X, X^*} - \ip{\delta w}{A \delta w}_{X, X^*}\\
		&\leq - \ip{\delta \lambda}{E \delta x}_{W^*, W} + \ip{\delta \beta_1}{\delta x}_{X, X^*} - 2 \ip{\delta v}{A \delta w}_{X, X^*} - \ip{\delta w}{A \delta w}_{X, X^*} \nonumber \\
		&\leq   \tilde k \norm{\delta \lambda}_{\underline{W}^*}\norm{\delta \beta_2}_W + \norm{\delta \beta_1}_{X^*} \norm{\delta x}_{X} + \norm{A} \norm{\delta w}_{X} \left( 2 \norm{\delta v}_X + \norm{\delta w}_X \right) \nonumber\\
		&\leq \tilde k k ( \norm{\delta w}_X + \norm{\delta v}_X + \norm{\delta \beta}_{\underline X^*\times W} ) \norm{\delta \beta_2}_W + ( \norm{\delta \beta_1}_{ X^*} + 2 \norm{A}\norm{\delta w}_X) (\norm{\delta v}_X + \norm{\delta w}_X), \nonumber
		\end{align}
where $\tilde k$ denotes the embedding constant of $\underline{W}^*$ into $W^*$.
		\nin   Using \eqref{eq:w_beta} and rearranging terms there exists a constant $k_2 > 0$ such that
		\begin{equation}
		\norm{\delta v}_X \leq k_2 \left( \norm{\delta \beta_1}_{\underline X^*} + \norm{\delta \beta_2}_W \right).
		\end{equation}

Applying \eqref{eq:w_beta} again this implies the existence of $ k_3$ such  that
		\begin{equation}\label{eq:aux11}
		\norm{\delta x}_X  \leq k_3 \left( \norm{\delta \beta_1}_{\underline X^*} + \norm{\delta \beta_2}_W \right) \text{ for all } (\beta_1, \beta_2) \in (\ul{X^*} \times W)\cap \tilde V.
		\end{equation}
Another application of \eqref{H4} and \eqref{eq:aux11} imply the existence of a constant  $k_4$
and a neighborhood $\hat V$ of the origin in $\ul{X^*} \times W$ such that
the desired Lipschitz stability estimate for $(\ul{\calT})^{-1}$
		\begin{equation}\label{eq:aux9}
			\norm{\delta x}_{\ul{X}} +  \norm{\delta \lambda}_{\ul{W^*}}   \leq k_4 \left( \norm{\delta \beta_1}_{\ul{X^*}} + \norm{\delta \beta_2}_W \right) \text{ for all } (\beta_1, \beta_2) \in  \hat V \subset \ul{X^*} \times W
		\end{equation}
		holds.
		\item As a consequence of the previous two steps $\ul{\calT}$ is strongly regular at $(x_0, q_0, \lambda_0)$. Together with step (i), Dontchev's theorem is applicable \cite[Theorem 2.4, Remark 2.5]{D:1995}, and \eqref{eq:thm_lc_1}, and \eqref{eq:thm_lc_2} follow.
		
		\item (Local solution to \eqref{Pq}) Now we show that there exists a neighborhood $\tilde{N}$ of $q_0$ such that for $q \in \tilde{N}$ the second order sufficient optimality condition is satisfied at $x(q)$, so that $x(q)$ is a local solution of \eqref{Pq} by eg. \cite[Theorem 2.12, p42]{IK:2008}. Due to \eqref{H2} and regularity of $f, e$ we obtain
		\begin{equation}\label{eq:h2_p}
		\calL''(x(q),q,\lambda(q))(h,h) \geq \frac{\kappa}{2} \norm{h}^2, \ \text{for all } h \in \text{ker } E, \ \text{if } q \in N(q_0).
		\end{equation}
		\nin Let us define $\ds E_q = (e_y(x(q),q))$ for $q \in N(q_0)$. By the surjectivity of $E_{q_0}$ and regularity of $e$ there exists a neighborhood $\tilde{N} \subset N(q_0)$ such that $E_q$ is surjective for all $q \in \tilde{N}$.  Here we also use continuity of
$q \mapsto e_y(x(q), q)$ from $P \to W$  at $q_0$, which follows from \eqref{H0} and the continuity of $q \to x(q)$ at $q_0$.  Consequently  exist $\delta_0, \gamma > 0$ such that
		\begin{equation}
		\calL''(x(q), q, \lambda(q))(h + z,h + z) \geq \delta_0 \norm{h + z}^2,\ \text{for all } h \in \text{ker } E, z \in X
		\end{equation}
		satisfying $\norm{z} \leq \gamma \norm{h}$ by \cite[Lemma 2.13, p43]{IK:2008}. Let us define the orthogonal projection onto ker$\ E_q$ given by $\ds P_{\text{ker } E_q} = I - E_q^*(E_qE_q^*)^{-1}E_q$. We choose $\tilde{N}$ so that
		\begin{equation*}
		\norm{E_q^*(E_qE_q^*)^{-1}E_q - E_{q_0}^*(E_{q_0} E_{q_0}^*)^{-1}E_{q_0}} \leq \frac{\gamma}{1 + \gamma}
		\end{equation*}
		for all $q \in \tilde{N}$.
		For $x \in \text{ker } E_q$, we have $x = h + z$ for $h \in \text{ker } E, \ z \in (\text{ker } E)^{\perp}$ and $\norm{x}^2 = \norm{h}^2 + \norm{z}^2$. Thus,
		\begin{equation*}
		\norm{z} \leq \norm{E_q^*(E_qE_q^*)^{-1}E_q x - E_{q_0}^*(E_{q_0} E_{q_0}^*)^{-1}E_{q_0} x} \leq \frac{\gamma}{1 + \gamma} \big( \norm{h} + \norm{z} \big)
		\end{equation*}
		and hence $\norm{z} \leq \gamma \norm{h}$. From \eqref{eq:h2_p} this implies
		\begin{equation*}
		\calL''(x(q), q, \lambda(q)) \geq \delta_0 \norm{x}^2, \ \text{for all } x \in \text{ker } E.
		\end{equation*}	
		\nin This concludes the proof. \qed
	\end{enumerate}
	
	\section{Differentiability of value function for optimal stabilization  subject to semi-linear parabolic equations.}\label{sec3}
	
	Here we describe the optimal control problems which we shall analyze and state the main results.
	
	\subsection{Notation}\label{sec3.1}
	
	\nin Let $\Omega$ be an open connected bounded subset of $\BR^d$ with dimension $d$, and a Lipschitz continuous boundary $\Gamma$. The associated space-time cylinder is denoted by $Q = \Omega \times (0,\infty)$ and the associated lateral boundary by  $\Sigma = \Gamma \times (0,\infty)$. We define the Hilbert spaces
	\begin{equation*}
	Y = L^2(\Omega), \quad V = H^1_0(\Omega), \text{ and }\quad U = L^2(0,\infty;\,\calU),
	\end{equation*}
	\nin where $\calU$ is a Hilbert space which will be identified with its dual.  Observe that the embedding $V \subset Y$ is dense and compact. Further $V \subset Y \subset V^*$, is a Gelfand triple. Here  $V^*$ denotes the topological dual of $V$ with respect to the pivot space $Y$. For any $T \in (0, \infty)$ we define the space
		\begin{equation*}
	W(0,T) = \bigg\{ y \in L^2(0,T;V); \ \frac{dy}{dt} \in L^2(0,T;V^*) \bigg\},
	\end{equation*}
	\nin endowed with the norm
	\begin{equation*}
		\norm{y}_{W(0,T)} = \left( \norm{y}_{L^2(0,T;V)}^2 + \norm{\frac{dy}{dt}}_{L^2(0,T;V^*)}^2 \right)^{\rfrac{1}{2}}.
	\end{equation*}
	\nin For $T = \infty$, we write $W_{\infty}$ and $I = (0,\infty)$. We further set $W^0_{\infty}=\{y\in W_{\infty}: y(0)=0 \}$. We also set
	\begin{equation*}
	W(T,\infty) = \bigg\{ y \in L^2(T,\infty;V); \ \frac{dy}{dt} \in L^2(T,\infty;V^*) \bigg\}.
	\end{equation*}	
	We shall frequently use that $W_\infty$ embeds continuously into $C([0,\infty),Y)$, see e.g. \cite[Theorem 4.2]{LM:1972} and that $\ds \lim_{t \to \infty} y(t) = 0$, for $y \in W_\infty$, see e.g. \cite{CK:2017}. The set of admissible controls $\Uad$ is chosen to be
	\begin{equation}\label{eq:Uad}
	U_{ad} \subset \{u \in U: \|u(t)\|_{\calU}\le \eta, \text{ for } a.e. \ t>0 \},
	\end{equation}
	where $\eta$ is a positive constant. We further set $\calU_{ad} = \{v \in \calU: \| v \|_{\calU}\le \eta\}$ and denote by $\ds \mathbb{P}_{\calU_{ad}}$ the projection of $\calU$ on $\calU_{ad}$. For this choice of admissible controls, the dynamical system can be stabilized for all sufficiently small initial conditions in $Y$, see  Corollary \ref{cor3.1} and Remark \ref{rem4.1}.\\
	
	\nin For $\delta > 0$ and $\bar y \in Y$, we define the open neighborhoods $B_Y(\delta) = \left\{ y \in Y: \norm{y}_Y < \delta \right\},$ and  $\quad B_Y(\bar y, \delta) = \left\{ y \in Y: \norm{y - \bar y}_Y < \delta \right\}$.

	\subsection{Problem formulation and assumptions.}
	
	We focus on the stabilization problem for an abstract semi-linear parabolic equation formulated as infinite horizon optimal control problem under control constraints:	
	\begin{subequations}\label{SLP}
		\begin{align}
		(\calP)  \qquad \V(y_0) = \min_{
			(y, u)\in W_\infty \times \Uad } \ J(y,u) &= \min_{
			 (y, u)\in W_\infty \times \Uad } \ \frac{1}{2} \int_{0}^{\infty} \norm{y(t)}^2_Y dt + \frac{\al}{2} \int_{0}^{\infty} \norm{u(t)}^2_{\calU} dt,
		\end{align}
		\nin subject to the semilinear parabolic equation
		\begin{align}
		y_t &= \calA y + \calF(y) + B u  \quad \text{ in } L^2(I; V^*), \label{1.1a}\\
		y(x,0) &= y_0 \quad \text{ in } Y. \label{1.1b}
		\end{align}	

	\nin Throughout $\calF$ is the substitution operator associated to a mapping  $\mathfrak{f}:\mathbb{R}\to \mathbb{R}$ so that $(\calF y)(t)=\mathfrak{f} (y(t))$. Sufficient conditions  which guarantee the existence of  solutions to  \eqref{1.1a}, \eqref{1.1b}, as well as solutions $(\bar{y}, \bar{u})$ to (\hyperref[SLP]{$\calP$}), for $y_0\in Y$  sufficiently small, will be given below. We shall also make use of the adjoint equation associated to an optimal state $\bar y$,  given by
		\begin{equation}
		-p_t - \calA^* p - \calF'(\bar{y})^* p = -\bar{y} \quad \text{ in  } L^2(I;V^*).
		\end{equation}
	\end{subequations}	
	Its adjoint state $p$ which will be considered in $L^2(I;V)$ or  in $W_\infty$. The following assumption will be essential.

	\subsubsection{Assumptions A.}\label{assump}
	\begin{itemize}
		\item[\textbf{A1}] The operator $\calA$ with domain $\calD(\calA)\subset Y$ and range in $Y$, generates a strongly continuous analytic semigroup $\ds e^{\calA t}$ on $Y$ and can be extended to  $\calA \in  \calL(V, V^*)$.		
		\item[\textbf{A2}] $B \in \calL(\calU,Y)$ and there exists a  stabilizing feedback operator $K \in \calL(Y,\calU)$ such that the semigroup $\ds e^{(\calA-BK)t}$ is exponentially stable on $Y$.
		\item[\textbf{A3}] The nonlinearity $\ds \calF: W_\infty \to L^2(I; V^*)$ is twice continuously Fr\'{e}chet differentiable, with  second Fr\'{e}chet derivative $\calF''$ bounded on bounded subsets of $W_\infty$,  and $\calF(0) = 0$.
		\item[\textbf{A4}] $\calF:W(0,T) \to L^1(0,T;\calH^*)$ is weak-to-weak continuous for every $T>0$, for some Hilbert space $\calH$ which embeds densely in $V$.
		
		Note that $\ds \left( L^1(0,T; \calH^*) \right)^*  = L^{\infty}(0,T; \calH)$, see \cite[Theorem 7.1.23(iv), p 164]{EE:2004}. Moreover, $L^{\infty}(0,T; \calH)$ is dense in $L^2(0,T; V)$, see \cite[Lemma A.1, p 2231]{MS:2017}.
		
		\item[\textbf{A5}] $\calF'(\bar{y})\in \calL{(L^2(I;V), L^2(I;V^*))}$.		
	\end{itemize}
	\begin{rmk}\label{rmk-fprime}
		The requirement that $\mathcal{F}(0) = 0$  in (\hyperref[assump]{A2}) is consistent with
		the fact that we focus on the stabilization problem with $0$ as steady state for \eqref{1.1a}. Without loss of generality we further assume that
		\begin{equation}\label{eq:fprime}
		\mathcal{F}'(0) = 0,
		\end{equation}
		which can always be achieved by making $\mathcal{F}'(0)$ to be perturbation of $\mathcal{A}$.
	\end{rmk}

	\begin{rmk}\label{rem3.2} Let us assume that (\hyperref[assump]{A3}) holds. Then in view of the fact that $\calF$ is a substitution  operator we have $[\calF'(y) v](t)= \mathfrak{f}'(y(t))v(t)$ for $y$ and $v$ in $W_\infty$,  and $\calF'(y)\in \calL(W_\infty,L^2(I;V^*))$. Its adjoint $[\calF'(y)^* v](t)= \mathfrak{f}'(y(t))v(t)$, for $v\in L^2(I;V)$, satisfying $\calF'(y)^*\in \calL(L^2(I;V),W_\infty^*)$. It has a natural restriction to an operator $ \ul{\calF'(y)^*} \in \calL(W_\infty,L^2(I;V^*))$. With (\hyperref[assump]{A3}) holding it is differentiable and $ [\ul{\calF'(y)}^*]'$ is a bilinear form on $W_\infty \times W_\infty$ with values in $L^2(I;V^*)$. -  For examples of functions $\calF$ which satisfy (\hyperref[assump]{A4}) we refer to see Section \ref{sec6}.
\end{rmk}
	
	\subsubsection{Abstract setup.} \label{Sec-abs_su}
	\nin Here we relate problem (\hyperref[SLP]{$\calP$}) to the abstract problem \eqref{Pq}, which is used with the following spaces:
	\begin{equation}\label{eq3.4}
	\begin{array}l
		X = W_{\infty} \times U, \; W = L^2(I;V^*) \times Y, \; P = Y, \;
		C = \Uad, \; X^* = W_{\infty}^* \times U, \; W^* = L^2(I;V) \times Y,\\[1.7ex]
  		\ul{X} = W_{\infty} \times (U \cap C(\bar I;\calU)), \quad \ul{X^*} = L^2(I;V^*) \times (U \cap C(\bar I;\calU)), \quad \ul{W^*} = \wti W_{\infty},
  \end{array}
\end{equation}
	\nin where \tcr{$I=(0,\infty)$,} and $ \widetilde \calW = \{(\varphi,\varphi(0)):\varphi \in W_\infty\}$, endowed with the norm of $W_{\infty}$. At times we identify $\widetilde \calW$ with $W_{\infty}$. We recall that the dual space of $\ds W_\infty = L^2(I;V) \cap W^{1,2}(I;V^*)$ is  $\ds W^*_\infty = L^2(I;V^*) + (W^{1,2}(I;V^*))^*$, endowed with the norm  $\ds \norm{z}_{ W^*_\infty} = \inf_{ z= z_1 + z_2} \norm{z_1}_{L^2(I;V^*)} + \norm{z_2}_{W^{1,2}(I;V^*)^*}$, where $\ds z_1 \in  L^2(I;V^*), z_2 \in (W^{1,2}(I;V^*))^*$.\\
	
	\nin To express \eqref{Pq} for the present case, we set $x = (y,u) \in W_{\infty} \times U$, and the parameter $q$ becomes the initial condition $y_0 \in Y$. Further $f: W_{\infty} \times U \longrightarrow \BR$ is given by
	\begin{equation}\label{eq:fyu}
	f(y,u) = \frac{1}{2}\int_{0}^{\infty} \norm{y(t)}^2_{Y} dt + \frac{\alpha}{2}\int_{0}^{\infty} \norm{u(t)}^2_{\calU} dt,
	\end{equation}
	and $e(x,q) = e(y,u,y_0)$ is
	\begin{equation}\label{def_eyu}
	e(y,u,y_0) = \bpm y_t - \calA y - \calF(y) - Bu \\ y(0) - y_0 \epm: W_{\infty} \times U \times Y \longrightarrow L^2(I; V^*) \times Y.
	\end{equation}
	\nin By (\hyperref[assump]{A3}) the mapping $e$ is Fr\'{e}chet differentiable with
respect to $\ds x = (y,u) \in W_{\infty} \times U$ and thus for $(y,u,y_0) \in W_{\infty} \times U \times Y$ we have
	\begin{equation}\label{def_deyu}
	e'(y,u,y_0)(v,w) = \bpm v_t - \calA v - \calF'(y) v - B w\\ v(0, \cdot) \epm : W_{\infty} \times U \longrightarrow L^2(I; V^*) \times Y.
	\end{equation}	
	\nin The Lagrange functional $\calL: W_\infty \times U \times Y \times L^2(I;V) \times Y \longrightarrow \BR$ corresponding to our optimal control problem is given by
	\begin{equation*}
	\calL(y, u, y_0, p,p_1) = J(y,u) + \int_0^{\infty} \ip{p}{y_t - \calA y - \calF(y) - Bu}_{V, V^*}dt + (p_1,y(0) - y_0)_Y,
	\end{equation*}
	where  $(p,p_1) \in L^2(I;V) \times Y$ corresponds to the abstract Lagrange multiplier $\lambda \in W^*$.\\
	
	\nin In the remainder of this subsection we specify the mappings $\calT$ and $\ul{\calT}$ for problem (\hyperref[SLP]{$\calP$}). This will facilitate the proofs of the main results further below.\\

\nin At first we take a closer look to the adjoint $\ds  \wti E^* := e'(y,u,y_0)^* \in \calL(L^2(I;V) \times Y,W^*_\infty \times U)$ at a generic element $(y,u,y_0) \in W_{\infty} \times U \times Y$. It is characterized by the property that for all $(v,w) \in W_\infty \times U,\ (p,p_1) \in L^2(I;V) \times Y$ we have	
	\begin{equation*}
		\ip{\wti{E}(v,w)}{(p,p_1)}_{L^2(I;V^*) \times Y, L^2(I;V) \times Y} =
\ip{v}{\wti{E}_1^*(p,p_1)}_{W_\infty, W^*_\infty} + \ipp{w}{\wti{E}_2^*(p,p_1)}_{U},
	\end{equation*}
	\nin where
	\begin{equation*}
		\ip{v}{\wti{E}_1^*(p,p_1)}_{W_\infty, W^*_\infty} =
		\ip{v_t - \calA v  -\calF'(y)v}{p}_{L^2(I;V^*),L^2(I;V)} + \ipp{v(0)}{p_1}_Y,
	\end{equation*}
	\nin and
	\begin{equation*}
		\ipp{w}{\wti{E}_2^*(p,p_0)}_{U} =  - \ipp{w}{B^*p}_U.
	\end{equation*}
	If for some $\tilde \beta_1 \in L^2(I;V^*)$ the pair $(p,p_1) \in L^2(I;V) \times Y$ is a solution to $\tilde E_1^*(p,p_1)= \tilde \beta_1 $ then for all $v \in W_{\infty}$:
	\begin{equation}\label{eq:aux12}
	\ip{v_t - \calA v -\calF'(y)v}{p}_{L^2(V^*),L^2(V)} + \ipp{v(0)}{p_1}_Y - \ipp{w}{B^*p}_U
	=  \langle \tilde \beta_1,v\rangle_{L^2(I;V^*),L^2(I;V)}.
	\end{equation}
	\nin Now we assume that $\calF'(y)$ is not only an element of $\ds \calL(W_{\infty},L^2(I;V^*))$ but rather that it can be extended to an operator ${\calF'(y)} \in \calL(L^2(I;V),L^2(I;V^*))$. This is guaranteed by (\hyperref[assump]{A5}) at minimizers $\bar y$. Then \eqref{eq:aux12} implies that $p
\in W_{\infty}$, and hence $p_1 \in C(I;Y)$ and $p_1 = p(0)$, see Proposition \ref{prop:adj}. In particular $(p,p_1) = (p,p(0)) \in \wti W_{\infty}$, and \eqref{eq:aux12}	can equivalently be expressed as
	\begin{equation}\label{vEtd_1}
		\ip{v}{\wti{E}_1^*(p,p_1)}_{L^2(I;V),L^2(I;V^*)} = \ip{v_t - \calA v -\calF'(y)v}{p}_{L^2(I;V^*),L^2(I;V)} = \ip{v}{\wti \beta_1}_{L^2(I;V),L^2(I;V^*)},
	\end{equation}
	\nin for all $v \in L^2(I;V)$, where we assumed that $\wti \beta_1 \in L^2(I;V^*)$. Conversely, of course, if $p \in \wti W_{\infty}$, then $\ds \wti E^*_1 (p,p(0)) = -p_t - \calA^*p - \calF'(y)^*p \in L^2(I;V^*)$.\\
	
	\nin From now on, let $q_0 = \bar{y}_0$ denote a reference (or nominal) parameter with associated solution $x_0 = (\bar{y}, \bar{u})$. In Proposition \ref{prop:adj} we shall argue that the regular point condition Assumption \eqref{H1} is satisfied and that consequently there exists a Lagrange multiplier $(\bar{p}, \bar{p}_1)$ such that the pair $(x_0, \lambda_0) = (\bar{y}, \bar{u}, \bar{p},\bar{p}_1)$ satisfies \eqref{eq:rp_h1}. Moreover, it will turn out that $\bar p \in W_\infty$, $\bar p_1 = p(0)$, and that $\bar u \in U \cap C(I;\calU)$. For convenience let us present \eqref{eq:rp_h1} for the present case
	\begin{equation}
	\begin{aligned}
	0 \in \begin{cases}
	\bar{y} + E^*_1(\bar p, \bar p(0)),\\
	\alpha \bar u - B^* \bar p + \partial \mb{I}_{\Uad}(\bar u),\\
	\bar y_t - \calA \bar y - \calF (\bar y) - B \bar u,\\
	{\bar y(0)} - y_0,
	\end{cases}
	\end{aligned}
	\end{equation}
	\nin where $\ds E = \bpm E_1 \\ E_2 \epm = e'(\bar y, \bar u, \bar y_0)$. We stress that while the Lagrange multiplier $\bar p$ belongs to $W_{\infty}$, the operator $E^*_1$ in \eqref{vEtd_1} is still considered as an element of $\calL(L^2(I;V) \times Y, W^*_{\infty})$.\\

	\nin We are now prepared to specify the multivalued operators
	\begin{align}
		\calT:& W_{\infty} \times U \times L^2(I;V) \times Y \longrightarrow W_{\infty}^* \times U \times L^2(I;V^*) \times Y, \quad \text{and} \\
		\ul{\calT}:& W_{\infty} \times (U \cap C(I;\calU)) \times \wti W_{\infty} \longrightarrow L^2(I;V^*) \times (U \cap C(I;\calU)) \times L^2(I;V^*) \times Y,
	\end{align}
corresponding to \eqref{mult_map} and \eqref{mult_map_b} by
	\begin{equation}\label{def:T}
		\calT \bpm y \\ u \\ p \\ p_1 \epm = \bpm  E^*_1(p,p_1) + y - [\calF'(\bar{y})^*\bar p]'(y - \bar{y})\\
		\alpha u - B^* p\\
		y_t - \calA y - Bu - \calF'(\bar{y})(y - \bar{y}) - \calF(\bar{y})\\
		y(0) - {\bar{y}_0} \epm  + \bpm 0 \\ \partial \mb{I}_{\Uad}(u) \\ 0 \\ 0 \epm,
	\end{equation}
and
	\begin{equation}\label{def:ulT}
		\ul{\calT} \bpm y \\ \ul u \\ \ul p \\ \ul p(0) \epm = \bpm - \ul p_t
- \calA^* \ul p - \ul {\calF'(\bar{y})^*}\, \ul p + y - [\ul{\calF'(\bar{y})^*}\,\ul{\bar{p}}]'(y - \bar{y})\\
		\alpha \ul u - B^* \ul p\\
		y_t - \calA y - B\ul{u} - \calF'(\bar{y})(y - \bar{y}) - \calF(\bar{y})\\
		y(0) - {\bar{y}_0} \epm  + \bpm 0 \\ \ul{\partial \mb{I}_{\Uad}}(\ul{u}) \\ 0 \\ 0 \epm,
	\end{equation}
	\nin where
	\begin{equation}
		\ul{\partial \mb{I}_{\Uad}}(\ul{u}) = \left\{ \wti u \in U \cap C(I;\calU): \ \ipp{\wti u(t)}{v - \ul{u}(t)}_\calU \leq 0, \ \forall t \in I, \ v \in B_{\eta}(0) \right\},
	\end{equation}
	with $\ds B_{\eta}(0) = \left\{ v \in \calU; \norm{v}_{\calU} \leq \eta \right\} $. In \eqref{def:ulT}, we underline the elements which are taken from different domains when compared to \eqref{def:T}. The range of the first two coordinates of $\ul \calT$ is smaller than that of $\calT$. Accordingly we can make use of \eqref{vEtd_1} when moving from the first row of \eqref{def:T} to the first row of \eqref{def:ulT}.\\

	\nin For convenience of the subsequent work, we recall that the strong regularity condition introduced below \eqref{mult_map_b} requires us  to find  neighborhoods  of $0$ and $(\bar{y}, \bar{u}, \bar{p},\bar{p}(0)) $  of the form $\hat{V} \subset L^2(I;V^*) \times (U \cap C(\ol{I};\calU)) \times L^2(I;V^*) \times Y$ and $\hat{U} \subset W_{\infty} \times (U \cap C(\ol{I};\calU)) \times \widetilde{W}_{\infty}$, such that for all $\bs{\beta} = (\beta_1, \beta_2, \beta_3, \beta_4) \in \hat V$ the equation
	\begin{equation}\label{def:T_b}
		\ul{\calT} \left(y, u, \ul p, \ul p(0) \right)^T = \left( \beta_1, \beta_2, \beta_3, \beta_4 \right)^T,
	\end{equation}	
	admits a unique solution $(y,\ul{u},\ul{p},\ul{p}(0)) \in \hat U$ depending Lipschitz-continuously on $\bs{\beta}$.

	\begin{rmk} We observe that as a consequence of (\hyperref[assump]{A3}) and Remark \ref{rem3.2} the operator  $\ul{\calT}$ is continuous.
	\end{rmk}	
	
	\nin Subsequently we shall frequently refrain from the underline-notation since the meaning should be clear from the context.
	
	\subsection{Main Theorems.}
	
	\nin In this subsection, we present the main theorems of this paper. The first theorem asserts  local continuous differentiability of the value function $\V$ w.r.t. $y_0$, with $y_0$ small enough. The second theorem establishes that $\V$ satisfies the HJB equation in the classical sense. The proof of the first theorem is based on Theorem \ref{thm:lip_con}. It will be given in Section \ref{sec4} below. For this purpose it will be shown that assumptions \hyperref[assump]{A} imply \eqref{H0}-\eqref{H4}. Moreover we need to assert the underlying assumption that problem (\hyperref[SLP]{$\calP$}) is well-posed. This will lead to a smallness assumption on the initial states $y_0$. Consequently it would suffice to assume that (\hyperref[assump]{A3}) and (\hyperref[assump]{A4}) only hold locally in the neighborhood of the origin.  Concerning (\hyperref[assump]{A5}) observe that it is not implied by (\hyperref[assump]{A3}). It  is vacuously satisfied for $\bar y=0$, which is the case for $y_0=0$, since then $\mathcal{F}'(0) = 0$, see \eqref{eq:fprime}.\\

 	\nin We invoke Theorem \ref{thm:lip_con} to assert the Lipschitz  continuity  of the state, the adjoint state, and the control with respect to the initial condition $y_0 \in Y$ in the neighborhood of  a locally optimal solution $(\bar y, \bar u)$ corresponding to a sufficiently small reference initial state $\bar y_0$. This will imply the differentiability of the value function associated to local minima. We shall refer to the value function associated to local minima as 'local value function'.

	\begin{thm}\label{thm-CD}
		Let the assumptions (\hyperref[assump]{A}) hold. Then associated to each local solution $(\bar y(y_0), \bar u(y_0))$ of (\hyperref[SLP]{$\calP$}) there exists a neighborhood of $U(y_0)$ such that the local value function   $\calV: U(y_0)\subset Y \to \BR$ is continuously differentiable, provided that $y_0$ is sufficiently close to the origin in $Y$.
\end{thm}

	\nin To obtain a HJB equation we require additionally that $t \to (\calF (\bar y))(t)$ is continuous with values in $Y$  for global  solutions $(\bar y, \bar u)$ to (\hyperref[SLP]{$\calP$}), with $y_0\in \calD(\calA)$. In view of the fact that  for $y_0\in V$ we can typically expect that the solutions of semilinear parabolic equations satisfy  $y \in L^2(I;\calD(\calA)) \cap W^{1,2}(I;Y) \subset C([0,\infty),V)$ this is not a restrictive assumption beyond that what is already assumed in (\hyperref[assump]{A3}).		
	
	\begin{thm}\label{thm-HJB}
	Let the assumptions (\hyperref[assump]{A}) hold, and let $(\bar y(y_0), \bar u(y_0))$
	denote a global solution  of (\hyperref[SLP]{$\calP$}),  for $y_0\in \calD(\calA)$ with  sufficiently small  norm in $Y$. Assume that there exists $T_{y_0}>0$ such that $\calF (\bar y)\in C([0,T_{y_0});Y)$. Then the following Hamilton-Jacobi-Bellman equation holds at $y_0$:
			\begin{equation}
			\V'(y)(\calA y + \calF(y)) + \frac{1}{2} \norm{y}^2_Y + \frac{\alpha}{2} \norm{ \mathbb{P}_{\mathcal{U}_{ad}} \left(\frac{1}{\alpha}B^*\V'(y) \right)}^2_Y + \left \langle B^* \V'(y),\mathbb{P}_{\mathcal{U}_{ad}} \left(\frac{1}{\alpha}B^*\V'(y) \right) \right\rangle_Y = 0.
			\end{equation}
Moreover the optimal feedback law is given by
			\begin{equation}
			{\bar u}(0) =  \mathbb{P}_{\mathcal{U}_{ad}} \left(\frac{1}{\alpha}  B^*\V'({\bar y}(0)) \right).
			\end{equation}
		\end{thm}
\nin The condition on the smallness of $y_0$ will be discussed in Remark \ref{rmk-Vhat} below. Roughly it involves well-posedness of the optimality system and second order sufficient optimality at local solutions. A more detailed,  respectively  stronger statement of Theorem  \ref{thm-CD} and Theorem \ref{thm-HJB}, will be given in Theorem  \ref{thm-CD-r} and Theorem  \ref{thm-HJB-r} below. The regularity assumptions $\calF (\bar y)\in C([0,T_{y_0});Y)$ of Theorem \ref{thm-HJB} will be addressed in Section \ref{sec6}.

	\section{Proof of Theorem \ref{thm-CD}.}\label{sec4}
	
	\nin In this section we give the proof for Theorem \ref{thm-CD}. Many of the technical difficulties arise from the fact that we are working with an infinite horizon optimal control problem. In this respect we can profit from techniques which were developed in \cite{BKP:2019}, which, however, do not include the case of constraints on the norm. Throughout we assume that assumptions (\hyperref[assump]{A1}) - (\hyperref[assump]{A4}) hold.
	
	\subsection{Well-posedness of problem (\hyperref[SLP]{$\calP$}).}
	
	\nin Here we prove  well-posedness for (\hyperref[SLP]{$\calP$}) with small initial data.  First, we recall two consequences of the assumption that $\calA$ is the generator of an analytic semigroup.
	\begin{consequence}\label{cons:1}
		Since $\calA$ generates a strongly continuous analytic semigroup on $Y$, there exist $\rho \geq 0$ and $\theta > 0$ such that
		\begin{equation*}
		\ip{(\rho I - \calA)v}{v}_{V^*,V} \geq \theta \norm{v}^2_V
		\end{equation*}
		\nin See \cite[Part II, Chaptor 1, p 115]{BPDM:2007}, \cite[Theorem 4.2, p14]{P:2007}.
	\end{consequence}
	
	\begin{consequence}\label{cons:2}
		For all $y_0 \in Y, f \in L^2(0,T; V^*)$, and $T > 0$, there exists a unique solution $y \in W(0,T)$ to
		\begin{equation}\label{eq3.2}
		\dot{y} = \calA y + f, \quad y(0) = y_0.
		\end{equation}
		\nin Furthermore, $y$ satisfies
		\begin{equation}
		\norm{y}_{W(0,T)} \leq c(T) \Big( \norm{y_0}_{Y} + \norm{f}_{L^2(0,T; V^*)} \Big),
		\end{equation}
		for a continuous function $c$. Assuming that $y \in L^2(0, \infty; Y)$,  consider the equation
		\begin{equation*}
		\dot{y} = \underbrace{(\calA - \rho I)y}_{\calA_{\rho}} + \underbrace{\rho y + f}_{f_{\rho}}, \quad y(0) = y_0,
		\end{equation*}
		\nin where $f_{\rho} \in L^2(I; V^*)$. Then the  operator $\calA_{\rho}$ generates a strongly continuous  analytic semigroup on $Y$ which is exponentially stable, see \cite[p 115, Theorem II.1.2.12]{BPDM:2007}. It follows that $y \in W_{\infty}$, that there exists $M_{\rho}$ such that
		\begin{equation}
		\norm{y}_{W_{\infty}} \leq M_{\rho} \Big( \norm{y_0}_{Y} + \norm{f_{\rho}}_{L^2(I; V^*)} \Big),
		\end{equation}
		\nin and that $y$ is the unique solution to \eqref{eq3.2}  in $W_{\infty}$, see \cite[Section 2.2]{BKP:2019} .
	\end{consequence}
	
	\begin{lemma}\label{lemma-Lip}
		There exists a constant $C > 0$, such that for all $\delta < (0,1]$ and for all $y_1$ and $y_2$ in $W_{\infty}$ with $\ds \norm{y_1}_{W_{\infty}} \leq \delta$ and $\ds \norm{y_2}_{W_{\infty}} \leq \delta$, it holds that
		\begin{equation}
		\norm{\calF(y_1) - \calF(y_2)}_{L^2(I; V^*)} \leq \delta C \norm{y_1 - y_2}_{W_{\infty}}.
		\end{equation}
	\end{lemma}
	\begin{proof}Let $y_1, y_2$ be as in the statement of the lemma. Using (\hyperref[assump]{A3}) and  Remark \ref{rmk-fprime} we obtain the estimate
		\begin{align*}
		&\norm{\calF(y_1) - \calF(y_2)}_{L^2(I,V^*)} \leq \int_0^1  \norm{\calF'(y_1+t(y_2-y_1)) - \calF'(0) }_{\calL(W_{\infty}, L^2(I,V^*))}\, dt \, \|y_2-y_1\|_{W_\infty}, \\
		&\le \int_0^1 \int_0^1 \norm{\calF''(s(y_1+t(y_2-y_1)))(ty_2+(1-t)y_1)}_{\calL(W_{\infty}, L^2(I,V^*))}\, ds dt \, \|y_2-y_1\|_{W_\infty}.
		\end{align*}
		\nin Now the claim follows by assumption (\hyperref[assump]{A3}).
	\end{proof}
	
	\begin{lemma}\label{exp-exi}
		Let $\ds \calA_s$ be the generator of an exponentially stable analytic semigroup $\ds e^{\calA_s t}$ on $Y$. Let $C$ denote the constant from Lemma \ref{lemma-Lip}. Then there exists a constant $M_s$ such that for all $y_0 \in Y$ and $f \in L^2(I; V^*)$ with
		\begin{equation*}
		\tilde \gamma = \norm{y_0}_Y + \norm{f}_{L^2(I; V^*)} \leq \frac{1}{4CM^2_s},
		\end{equation*}
		the system
		\begin{equation}\label{sys-As}
		y_t = \calA_s y + \calF(y) + f, \quad y(0) = y_0,
		\end{equation}
		has a unique solution $y \in W_{\infty}$, which  satisfies
		\begin{equation*}
		\norm{y}_{W_{\infty}} \leq 2 M_s \tilde \gamma.
		\end{equation*}
	\end{lemma}
	
	\nin With Lemma \ref{lemma-Lip} holding, this lemma can be verified in the same manner as \cite[Lemma 5, p 6]{BKP:2019}. In the following corollary we shall use Lemma \ref{exp-exi} with $\calA_s = \calA-BK$, and the constant corresponding to $M_s$ will be denoted by $M_K$. Further $\|\calI\|$ denotes the norm of the embedding constant of $W_\infty$ into $C(I;Y)$, $\|i\|$ is the norm of the embedding $V$ into $Y$, and we recall the constant $\eta$ from \eqref{eq:Uad}.

	\begin{clr} \label{cor3.1}
		For all $\ds y_0 \in Y$ with
		\begin{equation*}
			\norm{y_0}_Y \leq \min \left\{ \frac{1}{4CM^2_K}, \frac{\eta}{2M_K\norm{K}_{\calL(Y)}\norm{\calI}} \right\},
		\end{equation*}
		there exists a control $u \in {\Uad}$ such that the system
		\begin{equation}\label{eq3.6}
		y_t = \calA y + \calF(y) + Bu, \quad y(0) = y_0,
		\end{equation}
		has a unique solution $y \in W_{\infty}$ satisfying
		\begin{equation}\label{eq3.6a}
		\norm{y}_{W_{\infty}} \leq 2 M_K \|y_0\|_Y, \ \text{ and } \ \norm{u}_{U} \leq \|K\|_{\calL(Y,\,\calU)} \|\calI\| \|y\|_{W_\infty} \le 2M_K \|y_0\|_Y \|K\|_{\calL(Y,\,\calU)}\| \calI \|.
		\end{equation}
	\end{clr}
	
	\begin{proof}
		By Assumption (\hyperref[assump]{A2}), there exists $K$ such that $\calA - BK$ generates an exponentially stable analytic semigroup on $Y$. Taking $u = -Ky$, equation \eqref{eq3.6} becomes
		\begin{equation}\label{eq:3.7}
		y_t = (\calA - BK) y + \calF(y), \quad y(0) = y_0.
		\end{equation}
		Then by Lemma \ref{exp-exi} with $\tilde \gamma = \|y_0\|_Y$ there exists $M_K$ such that \eqref{eq:3.7} has a solution  $y\in W_\infty$ satisfying
		\begin{equation*}
		\norm{y}_{W_{\infty}} \leq 2 M_K  \|y_0\|_Y,
		\end{equation*}
	and thus the first inequality in \eqref{eq3.6a} holds. For every $t\in I$ we have
		\begin{equation}\label{eq3.9}
		\| u \|_U = \|Ky\|_U \le \|K\|_{\calL(Y,\,\calU)} \| y \|_Y \le \| K \|_{\calL(Y,\,\calU)} \| \calI \| \| y \|_{W_\infty} \le 2M_K \| y_0 \|_Y \| K \|_{\calL(Y,\,\calU)} \| \calI \|,
		\end{equation}
	and thus the second inequality in \eqref{eq3.6a} holds.We still need  to assert  that $u \in \Uad$. This follows from the second smallness condition on $\norm{y_0}_Y$ and \eqref{eq3.9}.
	\end{proof}	
	
	\begin{rmk} \label{rem4.1} In the above proof  stabilization was achieved by the feedback control $u = -Ky$. For this $u$ to be admissible it is needed that $\calU_{ad}$  has nonempty interior. The upper bound $\eta$
 could be allowed to be time dependent as long as it satisfies $\ds \inf_{t\ge 0}|\eta(t)|>0$.
	\end{rmk}
	
	\begin{clr}\label{clr-rho}
		Let $\ds y_0 \in Y$ and let $u \in \Uad$ be such that the system
		\begin{equation}\label{eq:3.6}
		y_t = \calA y + \calF(y) + Bu, \quad y(0) = y_0,
		\end{equation}
		has a unique solution $y \in L^2(I;Y)$. If
		\begin{align*}
		\gamma : = \norm{y_0}_Y + \norm{\rho y + Bu}_{L^2(I; V^*)}
		\leq \min \left\{ \frac{1}{4CM^2_{\rho}}, \frac{\eta}{2M_{\rho}\norm{K}_{\calL(Y)}\norm{\calI}} \right\},
		\end{align*}
		then $y \in W_{\infty}$ and it holds that
		\begin{equation*}
		\norm{y}_{W_{\infty}} \leq 2 M_{\rho} \gamma.
		\end{equation*}
	\end{clr}	
	\begin{proof}
		\nin Since $y \in L^2(I; Y)$, we can apply Lemma \ref{exp-exi} to the equivalent system
		\begin{equation*}
		y_t = (\calA - \rho I)y + \calF(y) + \tilde{f},
		\end{equation*}
		\nin where $\tilde{f} = \rho y + Bu$. This proves the assertion.
	\end{proof}
	\begin{lemma}\label{WPP}
		There exists $\delta_1 > 0$ such that for all $\ds y_0 \in B_Y(\delta_1)$, problem (\hyperref[SLP]{$\calP$}) possesses a solution $(\bar{y}, \bar{u}) \in W_{\infty} \times \Uad$. Moreover, there exists a constant $M > 0$ independent of $y_0$ such that
		\begin{equation}\label{eq:op_est}
		\max \big\{ \norm{\bar{y}}_{W_{\infty}}, \norm{\bar{u}}_U \big\} \leq M \norm{y_0}_{Y}.
		\end{equation}
	\end{lemma}
	\begin{proof}
		\nin The proof of this lemma follows with analogous argumentation as provided in \cite[Lemma 8]{BKP:2019}. Let us choose, $\delta_1 \leq \min \left\{ \frac{1}{4CM^2_K}, \frac{\eta}{2M_K\norm{K}_{\calL(Y)}\norm{\calI}} \right\}$, where $C$ as in Lemma  \ref{lemma-Lip} and $M_K$ denotes the constant from the Corollary \ref{cor3.1}. We obtain that for each $y_0 \in B_Y(\delta_1)$, there exists a control $u \in \Uad$ with associated state $y$ satisfying
		\begin{equation}\label{max_yu}
		\max \big\{ \norm{u}_{U}, \norm{y}_{W_{\infty}} \big\} \leq \tilde M \norm{y_0}_Y,
		\end{equation}
		\nin where $\ds \tilde M = 2M_K \text{ max } \big(1, \norm{i} \norm{K}_{\calL(Y,\,\calU)} \big)$. We can thus consider a minimizing sequence $\ds (y_n, u_n)_{n \in \mathbb{N}} \in W_\infty \times \Uad$ with $\ds J(y_n, u_n) \leq \frac{1}{2}M^2 \norm{y_0}^2_Y(1 + \alpha)$. For all $n \in \mathbb{N}$ that
		\begin{equation}\label{eq:aux1}
		\norm{y_n}_{L^2(I;Y)} \leq \tilde M \norm{y_0}_Y \sqrt{1 + \alpha} \quad \text{and} \quad \norm{u_n}_{L^2(I;\,\calU)} \leq \tilde M \norm{y_0}_Y \sqrt{\frac{1 + \alpha}{\alpha}}.
		\end{equation}
		\nin We set $\eta(\alpha,\tilde M) = \Big[ 1 + \tilde M \|i\|\ \sqrt{(1 + \alpha)} \Big( \rho + \frac{\norm{B}_{\calL(\calU, Y)}}{\sqrt{\alpha}} \Big) \Big]$. Then we have $\norm{y_0} + \norm{\rho y_n + B u_n }_{L^2(I;V^*)} \leq \eta(\alpha,\tilde M) \norm{y_0}_Y$. After further reduction of $\delta_1$, we obtain with $M_{\rho}$ from Corollary \ref{clr-rho}:
		\begin{equation*}
			\norm{y_0} + \norm{\rho y_n + B u_n }_{L^2(I;V^*)} \leq \frac{1}{4CM^2_{\rho}}.
		\end{equation*}	
		It  follows from this corollary that the sequence $\{y_n\}_{n \in \mathbb{N}}$ is bounded in $W_{\infty}$ with
		\begin{equation}\label{eq:aux2}
		\sup_{n \in \mathbb{N}} \norm{y_n}_{W_{\infty}} \leq 2 M_\rho (1+\eta(\alpha, \tilde M)) \norm{y_0}_{Y}.
		\end{equation}
		Extracting if necessary a subsequence, there exists $\ds (\bar{y}, \bar{u}) \in W_{\infty} \times U$ such that $\ds ({y}_n, {u}_n) \rightharpoonup (\bar{y}, \bar{u}) \in W_{\infty} \times U$, and $(\bar{y}, \bar{u})$ satisfies \eqref{max_yu}.\\
		
		\nin Let us prove that $(\bar{y}, \bar{u})$ is feasible and optimal.  Since $\Uad$ is weakly sequentially closed and $u_n \in \Uad$, we find {$\bar{u} \in \Uad$}.
		For each  fixed $T > 0$  and arbitrary $z \in L^{\infty}(0;T;\calH) \subset L^2(0,T;V)$, see (\hyperref[assump]{A4}), we have for all $\ds n \in \mathbb{N}$ that
		\begin{equation}\label{eq:aux3}
		\int_{0}^{T} \ip{\dot{y}_n(t)}{z(t)}_{V^*,V} dt = \int_{0}^{T} \ip{\calA y_n (t) + \calF(y_n(t)) + Bu_n(t)}{z(t)}_{V^*,V} dt.
		\end{equation}
		\nin Since $\dot{y}_n \rightharpoonup \dot{y}$ in $L^2(0,T; V^*)$, we can pass to the limit in the l.h.s. of the above equality. Moreover, since $\calA y_n \rightharpoonup \calA y$ in $L^2(0,T; V^*)$,
		\begin{equation*}
		\int_{0}^{T} \ip{\calA y_n (t)}{z(t)}_{V^*,V} dt \xrightarrow[n \rightarrow \infty]{} \int_{0}^{T} \ip{\calA \bar{y}(t)}{z(t)}_{V^*,V} dt.
		\end{equation*}
		\nin Analogously, we obtain that
		\begin{equation*}
		\int_{0}^{T} \ip{B u_n (t)}{z(t)}_{V^*,V} dt \xrightarrow[n \rightarrow \infty]{} \int_{0}^{T} \ip{B \bar{u}(t)}{z(t)}_{V^*,V} dt.
		\end{equation*}
		\nin If moreover $z \in L^{\infty}(0,T;\calH) \subset L^2(0,T;V)$, we use (\hyperref[assump]{A4}) to assert
		\begin{equation*}
		\int_{0}^{T} \ip{\calF (y_n(t)) - \calF(\bar{y}(t))}{z(t)}_{V^*,V} dt = \int_{0}^{T} \ip{\calF (y_n(t)) - \calF(\bar{y}(t))}{z(t)}_{\calH^*,\calH} dt
		\xrightarrow[n \rightarrow \infty]{} 0 .
		\end{equation*}
		\nin Thus we have for all $z \in L^{\infty}(0,T;\calH)$
		\begin{equation}\label{eq:aux4}
		\int_{0}^{T} \ip{\dot{y}(t) - \calA y(t) - Bu(t)}{z(t)}_{V^*,V} dt = \int_{0}^{T} \ip{\calF(y(t))}{z(t)}_{V^*,V} dt.
		\end{equation}
		\nin Since $\ds \dot{y} - \calA y - Bu \in L^2(0,T;V^*)$ and $L^{\infty}(0,T;\calH)$ is dense in $L^2(0,T;V)$ we conclude that \eqref{eq:aux4} holds for all $z \in L^2(0,T;V)$ and $T > 0$.	This yields $e(\bar{y},\bar{u}) = (0,0)$, and thus $(\bar{y},\bar{u})$  is feasible. By weak lower semicontinuity of norms it follows that $\ds J(\bar{y},\bar{u}) \leq \liminf_{n \rightarrow \infty} J(\bar{y}_n,\bar{u}_n)$, which proves the optimality of $(\bar{y},\bar{u})$, and \eqref{eq:op_est} follows from \eqref{eq:aux1}.
		\end{proof}
	
		\nin For the derivation of the optimality system for (\hyperref[SLP]{$\calP$}), we need the following lemma which is taken from  \cite[Lemma 2.5]{BKP:2018}.
	
	\begin{lemma}\label{lem-pur}
		Let $\ds G \in \calL(W_{\infty},L^2(I;V^*))$ such that $\ds \norm{G} < \frac{1}{M_K}$, where $\norm{G}$ denotes the operator norm of $G$. Then for all $\ds f \in L^2(I;V^*)$ and $y_0 \in Y$, there exists a unique solution to the problem:
		\begin{equation*}
		y_t = (\calA - BK)y(t) + (Gy)(t) + f(t), \quad y = y_0.
		\end{equation*}
		\nin Moreover,
		\begin{equation*}
		\norm{y}_{W_{\infty}} \leq \frac{M_K}{1 - {M_K \|G\|}} \left( \norm{f}_{L^2(I;V^*)} + \norm{y_0}_Y \right).
		\end{equation*}
	\end{lemma}
	
	\nin We close this section by deriving the optimality conditions for (\hyperref[SLP]{$\calP$}).
	
	\begin{prop}\label{prop:adj}
		Let the assumptions (\hyperref[assump]{A1}) - (\hyperref[assump]{A4}) hold. Then there exists $\delta_2 \in (0,\delta_1]$ such that each local solution $(\bar{y}, \bar{u})$ with $y_0 \in B_Y(\delta_2)$ is a regular point, i.e. \eqref{eq:rp_h1} is satisfied, and there exists an adjoint state $(\bar p, \bar p_1) \in L^2(I;V) \times Y$ satisfying		
		\begin{align}
		\ip{v_t - \calA v - \calF'(\bar{y})v}{\bar p}_{L^2(I;V^*), L^2(I;V)} + \ipp{v(0)}{\bar p_1}_Y + \ipp{\bar y}{\bar p}_{L^2(I;V)} &= 0, \quad \text{for all } v \in W_{\infty}, \label{adjP-1}\\
		\ip{\alpha \bar{u} - B^* \bar p}{u - \bar{u}}_{U} & \geq 0, \quad \text{for all } u \in \Uad. \label{adjP-3}
		\end{align}
		\nin If the assumption (\hyperref[assump]{A5}) is satisfied, then
		\begin{equation*}
			- \bar p_t - \calA^* \bar p - \calF'(\bar{y})^* \bar p = -\bar{y}, \text{  in  } L^2(I;V^*),
		\end{equation*}	
		and hence $\bar p \in W_{\infty}$ and
		\begin{equation}
			\lim_{t \rightarrow \infty} \bar p(t) = 0. \label{adjP-2}
		\end{equation}
		Moreover, there exists $\wti{M} > 0$, independent of $y_0 \in B_Y(\delta_2)$, such that
		\begin{equation}\label{adjP-4}
		\norm{\bar p}_{W_{\infty}} \leq \wti M \norm{y_0}_Y, \text{ and } u \in C(\ol{I}, \calU).
		\end{equation}
	\end{prop}
	\begin{proof}
		To verify the regular point condition, we evaluate $e$ defined in \eqref{def_eyu} at $(\bar{y}, \bar{u}, y_0)$. To check the claim on the range of $e'(\bar{y}, \bar{u}, y_0)$ we consider for arbitrary $\ds (r,s) \in L^2(I,V^*) \times Y$ the equation
		\begin{equation}\label{eq:zeq}
		z_t - \calA z - \calF'(\bar{y})z - B(w - \bar{u}) = r, \ z(0) = s,
		\end{equation}
		\nin for unknowns $(z,w) \in W_{\infty} \times {\Uad}$. By taking $w = -Kz \in U$ we obtain
		\begin{equation}\label{eq:aux13}
		z_t - (\calA - BK)z - \calF'(\bar{y})z + B \bar{u} = r, \ z(0) = s. \nonumber
		\end{equation}
		\nin {We apply Lemma \ref{lem-pur} to this equation with $\ds G = - \calF'(\bar{y})$ and $ f = r - B \bar{u}$. By Lemma \ref{WPP} and \eqref{eq:fprime} in Remark \ref{rmk-fprime} there exists $\delta_2 \in (0,\delta_1]$ such that $\|\calF'(\bar{y})\|_{\calL(W_{\infty},L^2(I;V^*))} \le \frac{1}{2} M_K$. Consequently by Lemma \ref{lem-pur}  there exists $\wti{M}$ such that}
		\begin{align}
		\norm{z}_{W_{\infty}} &\leq \wti{M} \big( \norm{r}_{L^2(I; V^*)} + \norm{s}_{Y} + \norm{B}_{ \calL(\calU, Y)} \norm{\bar{u}}_U\big) \nonumber \\
		&\leq \wti{M} \big( \norm{r}_{L^2(I; V^*)} + \norm{s}_{Y} + \norm{B}_{\calL(\calU, Y)} M \norm{y_0}_Y \big) \label{eq:z_est},
		\end{align}
		\nin with $M$ as in \eqref{eq:op_est}. We shall need to check whether $w = -Kz$ is feasible, which will be the case if $w(t) \leq \eta $ for a.e. $t \in I$. Indeed we have
		\begin{equation}\label{eq:w_est}
		\norm{w(t)}_{\mathcal U} \leq \norm{K}_{\calL(Y,\,\calU )} \norm{z(t)}_{Y} \leq \norm{K}_{\calL(Y,\,\calU)} \norm{\calI} \wti{M} \big( \norm{r}_{L^2(I; V^*)} + \norm{s}_{Y} + \norm{B}_{\calL(\calU, Y)} M \norm{y_0}_Y \big). \nonumber
		\end{equation}
		\nin Consequently, possibly after further reducing $\delta_2$, and choosing $\tilde{\delta} > 0$ sufficiently small we have
		\begin{equation}
		\norm{w}_{L^{\infty}(I;Y)} \leq \eta \text{ for all } \ds y_0 \in B_Y(\delta_2) \text{ and all } (r,s) \text{ satisfying } \norm{(r,s)}_{L^2(I;V^*) \times Y} \leq \tilde{\delta}.
		\end{equation}
		\nin Consequently the regular point condition is satisfied. Hence there exists a multiplier $\ds \lambda = (p,p_1) \in L^2(I;V) \times Y$ satisfying,
		\begin{equation}\label{eq:Lderi}
		\ip{\calL_y (\bar{y}, \bar{u}, y_0, \bar p, \bar p_1)}{v}_{L^2(I;V^*), L^2(I;V)} = 0, \quad \ip{\calL_u (\bar{y}, \bar{u}, y_0, \bar p, \bar p_1)}{u - \bar{u}}_{U} \geq 0, \ \forall u \in \Uad,
		\end{equation}
		\nin where
		\begin{equation*}
		\calL(y, u, y_0, p, p_1) = J(y,u) + \int_0^{\infty} \ip{p}{y_t - \calA y - \calF(y) - Bu}_{V, V^*}dt + \ip{p_1}{y(0) - y_0}_Y.
		\end{equation*}
		\nin This implies that \eqref{adjP-1} holds.\\
		
		\nin Now, if we impose the additional assumption (\hyperref[assump]{A5}), we have $\calF'(\bar{y})^* \bar p \in L^2(I;V^*)$. Thus
		$- \calA^* \bar p - \calF'(\bar{y})^* \bar p + \bar{y} \in L^2(I;V^*)$ and  the previous identity implies that $\bar p \in W_{\infty}$. Thus we derive
		\begin{equation*}
		- \bar p_t - \calA^* \bar p - \calF'(\bar{y})^* \bar p = -\bar{y} \ \text{in } L^2(I;V^*) \text{ and } \lim_{t \rightarrow \infty} \bar p(t) = 0,
		\end{equation*}
		and \eqref{adjP-1}-\eqref{adjP-2}. Testing the first identity in \eqref{eq:Lderi} with $v\in L^2(I;V)$ we also have $\bar p_1 = \bar p(0)$, which is well-defined since $\bar p \in W_{\infty} \subset C(I;Y)$. The second identity in \eqref{eq:Lderi} gives \eqref{adjP-3}. It remains to estimate $\ds \bar p \in W_\infty$.\\
		
		\nin Let $\ds r \in L^2(I;V^*)$ with $\ds \norm{r}_{L^2(I;V^*)} \leq \tilde{\delta}$ and consider
		\begin{equation}\label{eq:zeq_new}
		z_t - \calA z - \calF'(\bar{y})z - B (w-\bar{u}) = -r, \ z(0) = 0.
		\end{equation}
		\nin Arguing as in \eqref{eq:zeq}-\eqref{eq:z_est} there exists a solution to \eqref{eq:zeq_new} with $w = - Kz$ such that
		\begin{equation}\label{eq:z_est_new}
		\norm{z}_{W_{\infty}} \leq \wti{M} \big( \tilde{\delta} + \norm{B}_{\calL(\calU, Y)} M \norm{y_0}_Y \big) \leq \wti{M} \big( \tilde{\delta} + \norm{B}_{\calL(\calU,  Y)} M \delta_2 \big) =: C_1.
		\end{equation}
		\nin  From \eqref{eq:w_est} we have that $\ds \norm{w}_{L^{\infty}(I,\,\calU )} \leq \eta$. Let us now observe that
		\begin{align*}
		\ip{\bar p}{r}_{L^2(I,V), L^2(I,V^*)} &= \ip{\bar p}{-z_t + \calA z + \calF'(\bar{y})z}_{L^2(I,V), L^2(I,V^*)} + \ip{\bar p}{B(w-\bar{u})}_{L^2(I;Y)},\\
		&= \ip{\bar p_t + \calA^* \bar p + \calF'(\bar{y})^* \bar p}{z} + \ip{B^* \bar p}{w - \bar{u}}_U,
		\end{align*}
		\nin where we have used that $z(0) = 0$ and $\ds \lim_{t \rightarrow \infty} \bar p(t) = 0$, since $\bar p \in W_\infty$. We next estimate using \eqref{adjP-1}, \eqref{adjP-2} and \eqref{eq:z_est_new}
		\begin{align*}
		\ip{\bar p}{r}_{L^2(I,V), L^2(I,V^*)} &\leq \norm{\bar{y}}_{L^2(I,V^*)} \norm{z}_{L^2(I,V)} + \alpha \ip{\bar{u}}{w - \bar{u}}_U
		&\leq \left( \norm{\bar{y}}_{L^2(I,V^*)} + \alpha \norm{\bar{u}}_U \right) \left( C_1 + \eta + \norm{\bar{u}}_U \right).
		\end{align*}
		\nin By \eqref{eq:op_est}, this implies the existence of a constant $C_2$ such that
		\begin{equation*}
		\sup_{\norm{r}_{L^2(I,V^*)} \leq \tilde \delta} \ip{\bar p}{r}_{L^2(I,V), L^2(I,V^*)} \leq C_2 \norm{y_0}_Y,
		\end{equation*}
		and thus
		\begin{equation}\label{est:p}
		\norm{\bar p}_{L^2(I,V)} \leq \frac {C_2}{\tilde \delta} \norm{y_0}_Y, \quad \text{for all } y_0 \in B_Y(\delta_2).
		\end{equation}
		\nin  Now we estimate, again using (\hyperref[assump]{A5})
		\begin{align*}
		\norm{\bar p_t}_{L^2(I;V^*)} &\leq \norm{\calA^* \bar p + \calF'(\bar{y})^* \bar p - \bar{y}}_{L^2(I;V^*)}
		\leq C_3\norm{\bar p}_{L^2(I,V)} + C_4 \norm{\bar p}_{L^2(I,V)} + \norm{\bar{y}}_{L^2(I;V^*)}.
		\end{align*}
		\nin By \eqref{eq:op_est} and \eqref{est:p} we obtain $\norm{\bar p_t} _{L^2(I;V^*)} \leq C_5 \norm{y_0}_Y.$
		Combining this estimate with \eqref{est:p}  yields \eqref{adjP-4}. Finally, by \eqref{adjP-3} we find $\ds \bar u(t) = \mathbb{P}_{\calU_{ad}} \left( \frac{1}{\alpha} B^* \bar{p}(t) \right)$. Since $\bar p \in C(\bar I; Y)$ and $B^* \in \calL(Y,U)$ this implies that $u \in C(\bar I; \calU)$.
	\end{proof}
	
	\subsection{Verification of \eqref{H0}-\eqref{H3}.}
	
	\nin In this section we specialize the previously proved abstract results in Section \ref{Sec-abs_lip} to the semilinear parabolic setting. We start with the following lemma which shows that  assumptions \hyperref[assump]{A} imply \eqref{H0}-\eqref{H3}.
	
	\begin{lemma}\label{le4.7}
		Consider problem (\hyperref[SLP]{$\calP$}) with  assumptions (\hyperref[assump]{A1})-(\hyperref[assump]{A4}) holding. Then \eqref{H0}--\eqref{H00}, \eqref{H3} are satisfied for (\hyperref[SLP]{$\calP$}) uniformly for all $y_0 \in B_Y(\wti{\delta}_2)$ for some  $\ds \wti{\delta}_2 \in (0, \delta_2]$. If moreover, (\hyperref[assump]{A5}) holds, then \eqref{H000} holds as well.
	\end{lemma}
	\begin{proof} Throughout $y_0\! \in\! B_Y(\delta_2)$, $(\bar y, \bar u)$ denotes a local solution to (\hyperref[SLP]{$\calP$}), and $(p,p_1) \in L^2(I;V^*) \times Y$ the associated Lagrange multiplier.
		\begin{enumerate}[(i)]
			\item Verification of \eqref{H0}: The initial condition $y_0$ is our nominal reference parameter $q$. Lemma \ref{WPP} guarantees the existence of a local solution  $(\bar{y}, \bar{u}) \sim x_0$ to (\hyperref[SLP]{$\calP$})$\sim $ (\hyperref[Pq]{$P_{q_0}$}). Clearly $f$ defined in \eqref{eq:fyu} satisfies the required regularity assumptions.  Moreover  $e$ satisfies the regularity assumptions as a consequence of (\hyperref[assump]{A3}).
			\item Verification of \eqref{H1}: Proposition \ref{prop:adj} implies that $ (\bar{y}, \bar{u})$ is a regular point.
			\item Verification of \eqref{H2}: The second derivative of $e$ is given by
			\begin{equation}\label{sec_d_e}
			e''(\bar{y}, \bar{u}, y_0)((v_1,w_1), (v_2,w_2)) = \bpm \calF''(\bar{y}) (v_1, v_2) \\ 0 \epm, \quad \forall \ v_1, v_2 \in W_{\infty}, \ \forall w_1, w_2 \in U.
			\end{equation}
			\nin For the second derivative of $\calL$ w.r.t. $(y,u)$, we find
			\begin{multline}\label{sec_d_L.2}
			\calL''(\bar{y},\bar{u},y_0, \bar p, \bar p_1)((v_1,w_1),(v_2,w_2)) = \int_{0}^{\infty} \ipp{v_1}{v_2}_Y dt + \alpha \int_{0}^{\infty} \ipp{w_1}{w_2}_Y dt \\ + \int_{0}^{\infty} \ip{\bar p}{\calF''(\bar{y})(v_1, v_2)}_{V,V^*} dt.
			\end{multline}
			\nin By (\hyperref[assump]{A3}) for $\calF ''$ and Lemma \ref{WPP} , there exists $\wti{M}_1$ such that
			\begin{equation}
			\int_{0}^{\infty} \ip{\bar p}{\calF''(\bar{y})(v, v)}_{V,V^*} dt \leq \wti{M}_1 \| \bar p \|_{L^2(I;V)} \norm{v}^2_{W_{\infty}}, \quad \forall \ v\in W_\infty,
			\end{equation}
			for each solution $(\bar{y}, \bar{u})$ of (\hyperref[SLP]{$\calP$}) with $y_0 \in B_Y(\delta_2)$. Then we obtain
			\begin{equation}
			\calL''(\bar{y}, \bar{u}, y_0, \bar p, \bar p_1)((v,w),(v,w)) \geq \int_{0}^{\infty} \norm{v}_Y^2 dt + \alpha \int_{0}^{\infty} \norm{w}^2_{\calU} dt - \wti{M}_1 \| \bar p \|_{L^2(I;V)} \norm{v}^2_{W_{\infty}}.
			\end{equation}
			Now let $0\neq (v,w) \in \ker E \subset W_{\infty} \times \Uad$, where $E$ as defined in \eqref{def_deyu} is evaluated at $(\bar y,\bar u)$. Then,
			\begin{equation*}
			v_t  - \calA v - \calF'(\bar{y})v - Bw = 0, \quad v(0) = 0.
			\end{equation*}
			Next choose $\rho > 0$, such that the semigroup generated by $(\calA - \rho I)$ is  exponentially stable. This is possible due to (\hyperref[assump]{A1}). We equivalently write the system in the previous equation as,
			\begin{equation*}
			v_t  - (\calA - \rho I) v - \calF'(\bar{y})v - \rho v - Bw = 0, \quad v(0) = 0.
			\end{equation*}
			\nin Now, we invoke Lemma \ref{lem-pur} with ${\cal A}-BK$ replaced by $\calA - \rho I$, $\ds G = \calF'(\bar{y})$, and $f(t)=\rho v(t) + B w(t)$, and the role of the constant $M_K$ will now be assumed by a parameter $M_\rho$.  By selecting $\wti{\delta}_2 \in (0, \delta_2]$ such that $\norm{\bar{y}}_{W_{\infty}}$ sufficiently small, we can guarantee that $\ds \norm{\calF'(\bar{y})}_{\calL(W_{\infty}; L^2(I;V^*))} \leq \rfrac{1}{2 M_\rho}$, see \eqref{eq:op_est} and \eqref{eq:fprime} in Remark \ref{rmk-fprime}. Then the following estimate holds,
			\begin{equation*}
			\norm{v}_{W_{\infty}} \leq  2{M_\rho} \norm{v + Bw}_{L^2(I; V^*)}.
			\end{equation*}
			This implies that
			\begin{equation}\label{ker_vw}
			\norm{v}^2_{W_{\infty}} \leq  \wti{M}_2  ( \norm{v}^2_{L^2(I; Y)} + \norm{w}^2_{L^2(I; Y)} ).
			\end{equation}
			for a constant  $\wti{M}_2$ depending on $M_\rho, \|B\|$, and the embedding of $Y$ into $V^*$. 		These preliminaries allow the following lower bound on  $\calL''$:
			\begin{align}
			\calL''(\bar{y}, \bar{u}, y_0, \bar p, \bar p_1)((v,w),(v,w)) &\geq \int_{0}^{\infty} \norm{v}_Y^2 dt + \alpha \int_{0}^{\infty} \norm{w}^2_Y dt - \wti{M}_1 \|\bar p\|_{L^2(I;V)} \norm{v}^2_{W_{\infty}}, \nonumber \\
			\text{by ~ \eqref{ker_vw} ~} \geq \int_{0}^{\infty} \norm{v}_Y^2 + \alpha & \int_{0}^{\infty} \norm{w}^2_Y - \wti{M}_1  \wti{M}_2 \| \bar p \|_{L^2(I;V)} \left[ \norm{v}^2_{L^2(I; Y)} + \norm{w}^2_{L^2(I; Y)} \right], \nonumber\\
			= \left( 1 - \wti{M}_1 \wti{M}_2 \right. & \left. \| \bar p \|_{L^2(I;V)}  \right) \norm{v}^2_{L^2(I;Y)} + \left( \alpha - \wti{M}_1  \wti{M}_2 \| \bar p \|_{L^2(I;V)} \right) \norm{w}^2_{L^2(I;Y)}, \nonumber\\
			& \geq \tilde{\gamma} \left[ \norm{v}^2_{L^2(I;Y)} +  \norm{w}^2_{L^2(I;Y)} \right], \label{lb_L}
			\end{align}
			where $\ds \tilde{\gamma} = \min \left \{ 1 - \wti{M}_1 \wti{M}_2 \| \bar p \|_{L^2(I;V)} , \alpha - \wti{M}_1 \wti{M}_2 \| \bar p \|_{L^2(I;V)} \right \}$. By possible further reduction of $\tilde \delta_2$ it can be guaranteed that $\tilde \gamma >0$, see \eqref{est:p}. Then by \eqref{ker_vw}, we obtain,
			\begin{align*}
			\calL''(\bar{y}, \bar{u}, y_0, \bar p, \bar p_1)((v,w),(v,w)) &\ge \frac{\tilde{\gamma}}{2} \left[ \norm{v}^2_{L^2(I;Y)} +  \norm{w}^2_{L^2(I;Y)} \right] + \frac{\tilde{\gamma}}{2 \wti{M}_2} \norm{v}^2_{W_{\infty}},\\
			&\ge \frac{\tilde{\gamma}}{2 \wti{M}_2} \norm{v}^2_{W_{\infty}} + \frac{\tilde{\gamma}}{2} \norm{w}^2_{L^2(I;Y)}.
			\end{align*}
			\nin By selecting $\ds \bar{\gamma} = \min \left \{ \frac{\tilde{\gamma}}{2 \wti{M}_2}, \frac{\tilde{\gamma}}{2} \right \}$, we obtain the positive definiteness of $\calL''$, i.e.
			\begin{equation}
			\calL''(\bar{y}, \bar{u}, y_0, \bar p, \bar p_1)((v,w),(v,w)) \ge \bar{\gamma} \norm{(v,w)}^2_{W_{\infty} \times U}, \ y_0 \in B_Y(\wti{\delta}_2), \ (v,w) \in \text{ker } E.
			\end{equation}			
			Thus \eqref{H2} is satisfied.
			
			\item Verification of \eqref{H00}: It can easily be checked that $f'(y,u)$ can be extended to an element in $\ul{X^*} = L^2(I;V^*) \times (U \cap C(\bar I;\calU))$ for each $(y,u) \in X = W_\infty \times U$. We refer to Remark \ref{rem3.2} to show that the restriction of $e'(y,u,y_0)^*$ to $\ul{W^*}$ satisfies $\ul{e'(y,u,y_0)^*} \in \calL(\ul{W^*}, \ul{X^*}) = \calL(W_\infty \times Y,L^2(I;V^*)\times (U \cap C(\bar I;\calU)))$.
			
			\item Verification of \eqref{H3}: This is trivially satisfied.
			
			\nin Thus we have proved that assumptions (\hyperref[assump]{A1})-(\hyperref[assump]{A4}) imply \eqref{H0}-\eqref{H00}, and \eqref{H3} for all $\ds y_0 \in B_Y(\wti{\delta}_2)$.	
							
			\item Verification of \eqref{H000}: Here we use (\hyperref[assump]{A5}) and have $(p,p_1) = (p,p(0)) \in \wti W_{\infty}$. Observe that $\ul{\calL'}:W_{\infty} \times (U \cap C(\bar I;\calU)) \times Y \times W_\infty \to L^2(I;V^*) \times (U \cap C(\bar I;\calU))$ evaluated at $(\bar y, \bar u, y_0, p)$ is given by
			\begin{equation*}
				\ul{\calL'}(\bar y, \bar u, y_0, \bar p) = \bpm \bar y + \ul{e'( \bar y, \bar u, y_0)^*} \bar p \\ \alpha \bar u - B^* \bar p \epm = \bpm \bar y - \bar p_t - \calA^* \bar p - \calF'(\bar y)^* \bar p \\ \alpha \bar u - B^* \bar p \epm.
			\end{equation*}
			Further for  $(v.w)\in W_\infty \times U$ we have
			\begin{equation*}
				(\ul{\calL'})'(\bar y, \bar u, y_0, \bar p) (v,w) = \bpm v - [\calF'(\bar y)^*]'(\bar p,v) \\ \alpha w \epm \in L^2(I;V^*)\times (U \cap C(\bar I;\calU)).
			\end{equation*}	
		By (\hyperref[assump]{A3}) and Remark \ref{rem3.2} we have that $\ul{\calL'}$ and  $(\ul{\calL'})'$ are continuous as mappings from $W_{\infty} \times (U \cap C(\bar I;\calU)) \times Y \times W_\infty$ to $L^2(I;V^*) \times (U \cap C(\bar I;\calU))$, respectively to $\calL (W_\infty \times (U \cap C(\bar I;\calU)); L^2(I;V^*) \times (U \cap C(\bar I;\calU)))$.
	\end{enumerate}	
	\nin This proves the lemma.
	\end{proof}	

	\begin{rmk}\label{rmk-Vhat}
		Let us summarize our findings so far. There exists $\tilde \delta_2$ such that for each $y_0 \in B_Y(\tilde \delta_2)$ problem (\hyperref[SLP]{$\calP$}) posesses a solution $(\bar y, \bar u) \in W_{\infty} \times (U \cap C(\bar I;\calU))$, with an adjoint $\bar p \in \wti W_{\infty}$. Further (\hyperref[assump]{A1})-(\hyperref[assump]{A5}) imply \eqref{H0}-\eqref{H3} for (\hyperref[SLP]{$\calP$}) with $y_0 \in B_Y(\tilde \delta_2)$. As a consequence for each $y_0 \in B_Y(\tilde \delta_2)$ and each associated local solution $(\bar y, \bar u)$ there exists a neighborhood $\hat V$ of the origin in $\calY := L^2(I;V^*) \times (U \cap C(\bar I;\calU)) \times L^2(I;V^*) \times Y$ such that for each $\bs \beta \in \hat V$ there exists a unique solution $\ds \left(y_{(\bs \beta)}, u_{(\bs \beta)}, p_{(\bs \beta)}, {p_1}_{(\bs \beta)}  \right) \in W_{\infty} \times U \times L^2(I;V) \times Y$ to $\calT(y,u,p,p_1) = \bs \beta$, see step (ii) of the proof of Theorem  \ref{thm:lip_con} and $\ds \left( y_{(\bs \beta)}, u_{(\bs \beta)} \right)$ Lipschitz continuous w.r.t. $\bs \beta$, see \eqref{eq:aux11} in the proof of Theorem \ref{thm:lip_con}. --  To verify the remaining assumption \eqref{H4} we need to argue that $\left( u_{(\bs \beta)}, p_{(\bs \beta)} \right) \in (U \cap C(\bar I;\calU)) \times W_{\infty}$ and that $\ds \bs \beta \mapsto\left( y_{(\bs \beta)}, u_{(\bs \beta)}, p_{(\bs \beta)} \right)$ is Lipschitz continuous from $\hat V\subset \mathcal{Y}$  to $ W_{\infty} \times (U \cap C(\bar I;\calU)) \times W_{\infty}$ .
	\end{rmk}

\begin{rmk}\label{rmk-small}
Here we remark on the smallness assumption on $y_0$ expressed by $\delta_2$, respectively $\tilde \delta_2$. The condition $y_0 \in B_Y(\delta_2)$ guarantees the well-posedness of (\hyperref[SLP]{$\calP$}), existence and boundedness of adjoint states as expressed in Proposition \ref{prop:adj}. The additional condition $y_0 \in B_Y(\tilde\delta_2)$ implies that the second order optimality condition \eqref{H2} is satisfied, for each local solution associated to an initial condition $y_0\in B_Y(\delta_2)$. In the following we  formulate the results for all $y_0\in B_Y(\tilde\delta_2)$. Alternatively we could narrow down the claims to neighborhoods of single local solutions $(\bar y, \bar u)$ with  $y_0\in B_Y(\delta_2)$ and additionally assuming that the second order condition is satisfies at $(\bar y, \bar u)$. Concerning the second order condition itself, in some publications, see e.g. \cite{Gri:2004}, it is required to hold only for elements $x=(y,u) \in  \text{ker} E$ and $u = u_1-u_2$, with $u_1, u_2$  in $U_{ad}$. By a scaling argument it can easily be seen that this condition is equivalent to the one we use.

\end{rmk}

	\subsection{Verification of \eqref{H4} and Lipschitz stability of the linearized problem.}
	Throughout the remainder, we assume that  (\hyperref[assump]{A1})-(\hyperref[assump]{A5}) are satisfied and that $y_0 \in B_Y(\tilde \delta_2)$ so that Proposition \ref{prop:adj} and Lemma \ref{le4.7} are applicable. In the following, the triple $(y,u,p)$ refers to the solution $\calT(y,u,p,p_1) = \bs \beta$. Throughout without loss of generality, we also assume that $\hat V$ is bounded.
	
	
	\begin{lemma}\label{lem:beta_p}
		Let assumptions (\hyperref[assump]{A}) hold and let $(\bar y, \bar u)$, and $\bar p$ denote a local solution and associated adjoint state to (\hyperref[SLP]{$\calP$}) corresponding to an initial datum $y_0 \in B_Y(\tilde \delta_2)$. Then, possibly after further reduction of $\hat V$, the mapping $\bs \beta \mapsto p_{(\bs \beta)}$ is continuous from $\hat V$ to $W_{\infty}$.
	\end{lemma}
	\begin{proof}\ \\
		\nin \uline{Step 1:} For $\bs{\beta} \in \hat V$, with $\hat V$ as in Remark \ref{rmk-Vhat}, let $\ds (y_{(\bs \beta)}, u_{(\bs \beta)}, p_{(\bs \beta)}, {p_1}_{(\bs \beta)} )$ be the solution to $\calT(y,u,p,p_1) = \bs{\beta}$. As a consequence of (\hyperref[assump]{A5}) it is also a solution to $\ul \calT(y,u,p,p(0)) = \bs{\beta}$ with $p_{(\bs \beta)} \in W_{\infty}$. Thus the first two equations of this latter equality can be expressed as
		\begin{subequations}
			\begin{align}
			- \partial_t p_{(\bs \beta)} - \calA^* p_{(\bs \beta)} - \calF'(\bar{y})^*p_{(\bs \beta)} + y_{(\bs \beta)} - [\calF'(\bar{y})^*\bar{p}]'(y_{(\bs \beta)} - \bar{y}) &= \beta_1 \quad \text{ in } L^2(I;V^*), \label{adjTP_1}\\
			\ip{\alpha u_{(\bs \beta)} - B^* p_{(\bs \beta)} - \beta_2}{w - u_{(\bs \beta)}}_U &\geq 0 \quad \text{for all } w \in \Uad. \label{adjTP_2}
			\end{align}
		\end{subequations}
	\nin 
	\nin The above inequality is equivalent to $\ds u_{(\bs \beta)}(t) =  \mathbb{P}_{\mathcal{U}_{ad}} \left[ \frac{1}{\alpha} \left(B^* p_{(\bs{\beta})}(t) + \beta_2(t)\right) \right]$. Since $p_{(\bs \beta)} \in W_{\infty} \subset C(\bar I;Y)$ and $\beta_2 \in C(\bar I; \calU)$, it follows that $u_{(\bs \beta)} \in C(\bar I; \calU)$ for every $\bs \beta \in \hat V$.\\
	
	\nin \uline{Step 2: (Boundedness of $\{ p_{(\bs \beta)}: \bs \beta \in \hat V \}$).} Since $\hat V$ is assumed to be bounded, the discussion in Remark \ref{rmk-Vhat} shows that there exists a constant $M_1 > 0$ such that
	\begin{equation*}
		\norm{y_{(\bs \beta)}}_{W_{\infty}} + \norm{u_{(\bs \beta)}}_U \leq M_1 \quad \text{for all } \bs \beta \in \hat V.
	\end{equation*}
	\nin To argue the boundedness of $p_{(\bs \beta)}$, we use a similar technique as in the proof of Proposition \ref{prop:adj}. With $\tilde \delta$ as in the proof of that Proposition, $\bs \beta \in \hat V$, and $r \in R = \left\{ r \in L^2(I;V^*): \norm{r}_{L^2(I;V^*)} \leq \frac{\tilde \delta}{2} \right\}$ let $z$ denote the solution to
	\begin{equation}\label{eq:z}
	z_t - \calA z - \calF'(\bar{y})z - B \left(w - u_{(\bs \beta)} \right) = -r, \ z(0) = 0, \ w = -Kz.
	\end{equation}
	\nin From the proof of Proposition  \ref{prop:adj}, we know that there exists a constant $\tilde{M}$ such that
	\begin{align}
		\norm{z}_{W_{\infty}} &\leq \tilde{M} \left( \norm{r}_{L^2(I;V^*)} + \norm{B}_{\calL(\calU,Y)} \norm{u_{(\bs{\beta})}}_U \right), \nonumber\\
		&\leq \tilde{M} \left( \norm{r}_{L^2(I;V^*)} + \norm{B}_{\calL(\calU,Y)} \left( \norm{\bar u}_{U} + \norm{u_{(\bs{\beta})} - \bar u}_U \right) \right). \label{kk60}
	\end{align}
	\nin Consequently, we obtain with $M$ from \eqref{eq:op_est} for a.a. $t > 0$
	\begin{align*}
		\norm{w(t)}_{\calU} &\leq \norm{K}_{\calL(Y,\calU)} \norm{z(t)}_Y,\\
		&\leq \norm{K}_{\calL(Y,\calU)} \norm{\calI} \tilde{M} \left( \norm{r}_{L^2(I;V^*)} + \norm{B}_{\calL(\calU,Y)} \left( M \norm{y_0}_{Y} + \norm{u_{(\bs{\beta})} - \bar u}_U \right) \right),\\
		&\leq \norm{K}_{\calL(Y,\calU)} \norm{\calI} \tilde{M} \left( \frac{\tilde \delta}{2} + \norm{B}_{\calL(\calU,Y)} \left( M \tilde \delta_2 + \norm{u_{(\bs{\beta})} - \bar u}_U \right) \right),\\
		&\leq \eta + \norm{K}_{\calL(Y,\calU)} \norm{\calI} \tilde{M} \norm{B}_{\calL(\calU,Y)} \norm{u_{(\bs{\beta})} - \bar u}_U.
	\end{align*}
	\nin Due to the continuity of $\bs \beta \to u_{(\bs \beta)} \in U$, we obtain $\norm{w}_{L^{\infty}(I;Y)} \leq \eta$ possibly after further reduction of $\hat V$. Simultaneously, let us reduce $\hat V$ such that $\ds \norm{\frac{1}{\alpha} \beta_2}_{C(\bar I; \calU)} \leq \frac{\eta}{2}$ for all $\bs \beta \in \hat V$. Thus $w$ is feasible. Moreover we have that  $\norm{z}_{W_{\infty}} \leq C_1$ for a constant independently of $r \in R$ and $\bs \beta \in \hat V$. Due to \eqref{adjTP_1} and \eqref{adjTP_2}, we have
	\begin{align}\label{eq:p_r}
	\ip{p_{(\bs{\beta})}}{r}_{L^2(I;V), L^2(I;V^*)} &= \ip{p_{(\bs{\beta})}}{-z_t + \calA z + \calF'(\bar{y})z}_{L^2(I;V), L^2(I;V^*)} + \ip{B^*p_{(\bs{\beta})}}{w - u_{(\bs{\beta})}}_U, \nonumber\\
	\leq \ip{y_{(\bs{\beta})} - [\calF'(\bar{y})^*\bar{p}]'(y_{(\beta)} &- \bar{y}) - \beta_1}{z}_{L^2(I;V^*), L^2(I;V)} + \ip{\alpha u_{(\bs{\beta})} - \beta_2}{w - u_{(\bs{\beta})}}_{U},
	\end{align}
	where we also used the feasibility of $w \in \Uad$. Consequently
	\begin{multline*}
	\ip{p_{(\bs{\beta})}}{r}_{L^2(I;V), L^2(I;V^*)} \leq \norm{z}_{L^2(I;V)} \left( \norm{y_{(\bs{\beta})}}_{L^2(I;V^*)} + \norm{\beta_1}_{L^2(I;V^*)} + \norm{[\calF'(\bar{y})^*\bar{p}]'(y_{(\bs{\beta})} - \bar{y})}_{L^2(I;V^*)} \right)\\ + \left( \alpha  \norm{u_{(\bs{\beta})}}_{U} + \norm{\beta_2}_U \right)\norm{w - u_{(\bs{\beta})}}_U.
	\end{multline*}	
	\nin The right hand side is uniformly bounded for $\bs \beta$ in the bounded set $\hat V$ and w.r.t. $r \in R$. Hence taking the supremum w.r.t. $r \in R$ we verified that $\ds \left\{ \norm{p_{(\bs \beta)}}_{L^2(I;V^*)}: \bs \beta \in \hat V \right\}$ is bounded. Boundedness of $\ds \left\{ \norm{p_{(\bs \beta)}}_{W_{\infty}}: \bs \beta \in \hat V \right\}$ follows from \eqref{adjTP_1}.\\
	
	\nin \uline{Step 3: (Continuity of $p_{(\bs \beta)}$ in $W_{\infty}$).} Let $\{ \bs \beta_n \}$ be a convergent sequence in $\hat V$ with limit $\bs \beta$. Since $\ds \left\{ \norm{p_{(\bs \beta_n)}}_{W_{\infty}}: n \in \BN \right\}$ is bounded, there exists a subsequence $\{ \bs \beta_{n_k} \}$ such that $\ds p_{\left(\bs \beta_{n_k}\right)} \rightharpoonup \tilde{p}$ weakly in $W_{\infty}$ and strongly $L^2(0,T;Y)$ for every $T \in (0,\infty)$, see e.g. \cite[Satz 8.1.12, pg 213]{EE:2004}. Passing to the limit in the variational form of
	\begin{equation*}
		- \partial_t p_{(\bs \beta_{n_k} )}  - \calA^* p_{(\bs \beta_{n_k}) } - \calF'(\bar{y})^*p_{(\bs \beta_{n_k}) } + y_{(\bs \beta_{n_k}) } - [\calF'(\bar{y})^*\bar{p}]' \left(y_{(\bs \beta_{n_k} )} - \bar{y} \right) = \beta_{n_k},_1,
	\end{equation*}
	\nin we obtain
	\begin{equation}
	- \partial_t \tilde{p} - \calA^* \tilde{p} - \calF'(\bar{y})^*\tilde{p} + y_{(\bs \beta)} - [\calF'(\bar{y})^*\bar{p}]'\left( y_{(\bs \beta)} - \bar{y} \right) =  \beta_1.
	\end{equation}
	\nin Since the solution to this equation is unique, we have $\ds p_{(\bs \beta_n)} \rightharpoonup p_{(\bs \beta)}$ weakly in $W_{\infty}$. To obtain strong convergence, we set $\delta \bs \beta = \bs \beta_n - \bs \beta, \ \delta p = p_{(\bs \beta_n)} - p_{(\bs \beta)},$ and $\delta y = y_{(\bs \beta_n)} - y_{(\bs \beta)}$. Since $\ds p_{(\bs \beta)} \in W_{\infty}$ we have that $\ds \lim_{t \rightarrow \infty} p_{(\bs \beta)}(t) = 0$ in $Y$. Hence there exists $\hat T$ such that $\ds \frac{1}{\alpha} \norm{B^* p_{(\bs \beta)} (t) }_{\calU} \leq \frac{\eta}{4}$ for all $t \geq \hat T$, and by the choice of $\hat V$, we also have that
	\begin{equation*}
		\norm{u_{(\bs{\beta})}(t)}_{\calU} = \norm{ \mathbb{P}_{\mathcal{U}_{ad}} \left[ \frac{1}{\alpha} \left(B^* p_{(\bs{\beta})}(t) + \beta_2(t)\right) \right]}_{\calU} = \frac{1}{\alpha} \norm{B^* p_{(\bs{\beta})}(t)	+ \beta_2(t)}_{\calU} \leq \frac{3 \eta}{4},
	\end{equation*}
	\nin i.e. the constraint is inactive for $t \geq \hat T$. Let us estimate for $z \in W_{\infty}$,
	\begin{align*}
		\langle \delta p,  B( Kz - u_{( \bs \beta_{n_k})})&\rangle_{L^2(I;Y)} \leq \int_{0}^{\hat T} \ip{p_{(\bs \beta_{n_k})}(t) - p_{(\bs \beta)}(t)}{B( Kz(t) - u_{( \bs \beta_{n_k} )}(t) )}_Y dt \\
		 & \hspace{1.3cm} + \int_{\hat T}^{\infty} \ip{B^*(p_{(\bs \beta_{n_k})}(t) - p_{(\bs \beta)}(t))}{Kz(t)- u_{( \bs \beta_{n_k} )}(t) }_{\mathcal{U}}  \,  dt,\\
		 &\leq \int_{0}^{\hat T} \norm{B^*(p_{(\bs \beta_{n_k})}(t) - p_{(\bs \beta)}(t))}_Y\norm{Kz(t) - u_{( \bs \beta_{n_k} )}(t) }_{\mathcal{U}}dt\\
		 &\hspace{1cm} + \int_{\hat T}^{\infty} \ip{(\alpha u_{(\bs \beta_{n_k})}(t) - \beta_{n_k,2}(t)) - (\alpha u_{(\bs \beta)}(t) - \beta_2(t))}{ Kz(t)  - u_{(\bs \beta_{n_k})}(t)}_{\mathcal{U}}dt,\\
		 &\leq  \left( \norm{B}_{\calL(\calU,Y)} \norm{p_{(\bs \beta_{n_k})} - p_{(\bs \beta)}}_{L^2(0,\hat T; Y)} + \alpha \norm{u_{(\bs \beta_{n_k})} - u_{(\bs \beta)}}_U \right.\\
		 &\hspace{1.3cm} \left. + \norm{\beta_{n_k,2} - \beta_2}_U \right) \left( \norm{K}_{\calL(Y,\calU)} \norm{z}_{W_{\infty}} + \norm{u_{( \bs \beta_{n_k} )}}_U \right).
	\end{align*}
	\nin Let $R_1 = \left\{ r \in L^2(I;V^*): \norm{r}_{L^2(I;V^*)} \leq 1 \right\}$, and denote the solution to \eqref{eq:z} by $z = z_{(\bs \beta)}$ for $\bs \beta \in \hat V$. From the estimates in \eqref{kk60} there exists $M_2$ such that $\ds \norm{z_{(\bs \beta)}}_{W_{\infty}} \leq M_2$ for all $\bs \beta \in \hat V$, and $r \in R_1$.\\
	
	 \nin From \eqref{adjTP_1} we derive that
	\begin{equation}\label{eq:del_p}
	- \partial_t (\delta p) - \calA^* (\delta p) - \calF'(\bar{y})^*(\delta p) + (I - [\calF'(\bar{y})^*\bar{p})]')(\delta y) = \delta \beta_1 
	\end{equation}
	\nin holds in $L^2(I;V^*)$. Hence from \eqref{eq:z} we find for arbitrary $r \in R_1$
	\begin{equation*}
	\ip{\delta p}{r}_{L^2(I;V), L^2(I;V^*)} = \ip{(I - [\calF'(\bar{y})^*\bar{p}]')(\delta y) - \delta \beta_1}{z}_{L^2(I;V^*), L^2(I;V)} + \ip{\delta p}{B(Kz_{(\bs \beta)} - u_{(\bs \beta)})}_{L^2(0,\hat T; Y)},
	\end{equation*}
	\nin and thus for some $C_2 > 0$,
	\begin{align}
	\norm{\delta p}_{L^2(I;V)} &= \sup_{r \in R_1} \ \ip{\delta p}{r}_{L^2(I;V), L^2(I;V^*)} \nonumber \\
	&\leq C_2 \left( \norm{\delta y}_{W_{\infty}} + \|\ipp{\delta \beta_1}{\delta \beta_2}\|_{L^2(I;V^*) \times U} + \norm{\delta p}_{L^2(0,\hat T;Y)} + \norm{\delta u}_U \right). \label{est_p_b}
	\end{align}
	\nin Since $\norm{\delta y}_{W_{\infty}} \to 0,\ \norm{\delta p}_{L^2(0,\hat T;Y)} \to 0,\ \|{\ipp{\delta \beta_1}{\delta \beta_2}\|_{L^2(I;V^*) \times U}} \to 0$ for $n \to 0$ this implies that $\norm{\delta p}_{L^2(I;V)} \to 0$. Together with \eqref{eq:del_p} it follows that $\ds \lim_{n \rightarrow \infty} \norm{\delta p}_{W_{\infty}} = 0$.
	\end{proof}


	\begin{prop}\label{prop:est_p}
		Let assumptions (\hyperref[assump]{A}) hold and let $(\bar{y}, \bar{u})$, and $\bar p$ denote a local solution and associated adjoint state state to (\hyperref[SLP]{$\calP$}) corresponding to an initial condition $\bar y_0 \in B_Y(\tilde \delta_2)$. Then there exists $\varepsilon > 0$ and $C > 0$ such that  for all  $\bs{\hat{\beta}}$ and  $\bs{\beta} \in \hat{V} \cap B_{\calY}(\varepsilon)$
		\begin{multline}\label{est:p_dif}
		\norm{\hat{p}_{(\hat{\bs \beta})} - p_{(\bs \beta)}}_{W_{\infty}} + \norm{u_{(\hat{\bs \beta})} - u_{(\bs \beta)}}_{C(\bar I;\calU)}\!  \\
\leq C  \left( \norm{\hat{y}_{(\hat{\bs \beta})} - y_{(\bs \beta)}}_{W_{\infty}} + \norm{u_{(\hat{\bs \beta})} - u_{(\bs \beta)}}_{U} + \norm{\hat{\bs \beta} - \bs \beta}_{\calY} \right)
		\end{multline}
holds.
	\end{prop}
	\begin{proof}
		\nin As we described in Step 3 of the proof of Lemma \ref{lem:beta_p}, since $\bar p \in W_{\infty}$ and $\ds \lim_{t \rightarrow \infty} \bar p(t) = 0$ in $Y$, there exists $T > 0$ such that
		\begin{equation*}
		\frac{1}{\alpha}\norm{B^*\bar{p}(t)}_Y \leq \frac{\eta}{2}, \quad \forall t > T.
		\end{equation*}
		\nin Since $p(0) = \bar p$, and since by Lemma \ref{lem:beta_p}, $\ds \bs{\beta} \in \calY \mapsto p_{(\bs{\beta})} \in W_{\infty} \subset C(\bar I;Y)$ is continuous, there exists $\varepsilon > 0$ such that
		\begin{equation*}
		\frac{1}{\alpha} \norm{B^* p_{(\bs{\beta})}(t) + \beta_2(t)}_{Y} \leq \frac{\eta}{4}, \quad \forall t \geq T, \ \forall \bs{\beta} \in \hat V \cap B_{\calY}(\varepsilon).
		\end{equation*}
		Consequently the constraints are inactive for these parameter values,  i.e. we have
		\begin{equation}\label{eq:u_b}
		u_{(\bs{\beta})}(t) = \frac{1}{\alpha} \left[B^* p_{(\bs{\beta})}(t)
		+ \beta_2(t)\right], \ \norm{u_{(\bs{\beta})}(t)}_{Y} \leq \eta, \quad \forall t \geq T, \ \forall \bs{\beta} \in \hat V \cap B_{\calY}(\varepsilon).
		\end{equation}
		\nin We next treat separately the cases $[0,T)$ and $[T,\infty)$. We consider first the case $[T,\infty)$ and  set $(y,u,p) = \left( y_{(\bs{\beta})}, u_{(\bs{\beta})}, p_{(\bs{\beta})} \right)$, and $\left( \hat{y}, \hat{u}, \hat{p} \right) = \left( \hat{y}_{(\bs{\hat{\beta}})}, \hat{u}_{(\bs{\hat{\beta}})}, \hat{p}_{(\bs{\hat{\beta}})} \right)$. We shall use that
		\begin{equation*}
		\norm{\hat{p} - p}_{L^2(T,\infty; V)} = \sup_{\norm{r}_{L^2(T,\infty;
				V^*)} \leq 1} \int_T^{\infty} \ip{\hat{p}(t) - p(t)}{r(t)}_{V,V^*}dt.
		\end{equation*}
		\nin Let $z \in W(T,\infty)$ be such that,
		\begin{equation*}
		z_t - (\calA - BK)z - \calF'(\bar{y})z = r, \ z(T) = 0,
		\end{equation*}
		\nin From Lemma \ref{lem-pur}, see also the proof of Proposition \ref{prop:adj}, we know that there exists a constant $C_1 > 0$ such that $\ds \norm{z}_{W(T,\infty)} \leq C_1 \norm{r}_{L^2(T,\infty;V^*)}$. Then we can estimate
		\begin{align*}
		\norm{\hat{p} - p}_{L^2(T,\infty; V)} &= \sup_{\norm{r}_{L^2(T,\infty;V^*)} \leq 1} \int_T^{\infty} \ip{\hat{p} - p}{r}_{V,V^*} dt,\\
		&\leq \sup_{\norm{r} \leq 1} \int_T^{\infty} - \ip{(\hat{p}_t - p_t) + \calA^*(\hat{p} - p) + \calF'(\bar{y})^*(\hat{p} - p)}{z}_{V^*,V}  dt \nonumber\\
		& \hspace{7cm} + \sup_{\norm{r} \leq 1} \int_T^{\infty} \ip{B^*(\hat{p} -
			p)}{Kz}_{V,V^*} dt.
		\end{align*}
		\nin In the following, $C_i$ denote constants independent of $\bs{\hat{\beta}} \text{ and } \bs{\beta} \in \hat{V} \cap B_{\calY}(\varepsilon)$. From \eqref{adjTP_1} and \eqref{eq:u_b} we obtain, for $C_2 > 0$,
		\begin{multline*}
		\norm{\hat{p} - p}_{L^2(T,\infty; V)} \leq C_2 \sup_{\norm{r} \leq 1} \int_T^{\infty} \left[ \norm{\hat{y} - y}_{V^*} + \norm{[\calF'(\bar{y})^*\bar{p}]'(\hat{y} - y)}_{V^*} + \norm{\hat{\beta}_1 - \beta_1}_{V^*}  \right.\\
		\left. + \norm{\hat{\beta}_2 - \beta_2}_{U \cap C(I;\calU)} + \alpha \norm{B^*} \norm{K} \norm{\hat{u} - u} \right] \norm{z}_V dt.
		\end{multline*}
		\nin From (\hyperref[assump]{A3}) recall that $\ds \norm{[\calF'(\bar{y})^*\bar{p}]'(\hat{y} - y)}_{L^2(I;V^*)} \leq C_3 \norm{\hat{y} - y}_{W_{\infty}}$. This gives the following estimate for $C_4 > 0$,
		\begin{equation}\label{est:p_wti}
		\norm{\hat{p} - p}_{L^2(T,\infty; V)} \leq C_4  \left( \norm{\hat{y} - y}_{W_{\infty}} + \norm{\hat{u} - u}_{U} + \norm{\hat{\bs \beta} - \bs \beta}_{\calY} \right).
		\end{equation}
		By \eqref{adjTP_1} we have $\ds \left( \hat{p}_t - p_t \right) \in L^2(T, \infty;V^*)$. Then we obtain $\hat{p} - p \in W(T,\infty)$. Then there exists $C_5 > 0$ independent of $\bs{\hat{\beta}}\text{ and } \bs{\beta} \in \hat{V} \cap B_{\calY}(\varepsilon)$ such that,		
		\begin{equation}
		\norm{\hat{p} - p}_{W(T,\infty)} \leq C_5  \left( \norm{\hat{y} - y}_{W_{\infty}} + \norm{\hat{u} - u}_{U} + \norm{\hat{\bs \beta} - \bs \beta}_{\calY} \right).
		\end{equation}
		\nin By the embedding $\ds W(T,\infty) \subset C(T,\infty;Y)$, there exists a constant $C_6 > 0$:
		\begin{equation}\label{est:dif_p_con}
		\norm{\hat{p} - p}_{C([T,\infty);Y)} \leq C_6 \left( \norm{\hat{y} - y}_{W_{\infty}} + \norm{\hat{u} - u}_{U} + \norm{\hat{\bs \beta} - \bs \beta}_{\calY} \right).
		\end{equation}
		\nin Similarly, we estimate on $[0,T]$:
		\begin{equation}
		\norm{\hat{p} - p}_{L^2(0,T; V)} = \sup_{\norm{r}_{L^2(0,T; V^*)} \leq 1} \int_0^{T} \ip{\hat{p} - p}{r}_{V,V^*} dt.
		\end{equation}
		\nin Choose $z$ as
		\begin{equation*}
		z_t - \left( \calA z + \calF'(\bar{y})z \right) = r, \ z(0) = 0,
		\end{equation*}
		\nin Then there exists $\ds C_7 > 0$ such that $\ds \norm{z}_{W(0,T)} \leq C_7 \norm{r}_{L^2(0,T; V^*)}$ by Lemma \ref{lem-pur}. Note that $C_7$ depends on $T$, but $T$ is fixed. We obtain the following estimate,
		\begin{multline*}
		\norm{\hat{p} - p}_{L^2(0,T; V)} \leq \sup_{\norm{r} \leq 1} \int_0^T -
		\ip{(\hat{p}_t - p_t) + \calA^*(\hat{p} - p) + \calF'(\bar{y})^*(\hat{p} - p)}{z}_{V^*,V} dt
		\\+ \norm{\hat{p}(T) - p(T)}_Y\norm{z(T)}_Y.
		\end{multline*}
		\nin Then by a similar computation to that for  the $t \in [T, \infty)$ case, we obtain,
		\begin{equation}\label{est:p_ot}
		\norm{\hat{p} - p}_{L^2(0,T; V)} \leq C_8 \left( \norm{\hat{y} - y}_{W_{\infty}} + \norm{\hat{\beta}_1 - \beta_1}_{L^2(I;V^*)} \right) + \norm{\hat{p}(T) - p(T)}_Y\norm{z(T)}_Y.
		\end{equation}
		\nin By \eqref{est:dif_p_con} with $\norm{z(T)}_Y \leq C_9$, we obtain
		\begin{equation*}
		\norm{\hat{p}(T) - p(T)}_Y\norm{z(T)}_Y \leq C_7 C_9 \left( \norm{\hat{y} - y}_{W_{\infty}} + \norm{\hat{u} - u}_{U} + \norm{\hat{\bs \beta} - \bs \beta}_{\calY} \right).
		\end{equation*}
		\nin Combining this estimate with \eqref{est:p_wti} and \eqref{est:p_ot}, we obtain for some $C_{10} > 0$
		\begin{equation}\label{est:dif_p_T}
		\norm{\hat{p}_{(\hat{\bs \beta})} - p_{(\bs \beta)}}_{W_{\infty}} \leq C_{10} \left( \norm{\hat{y}_{(\hat{\bs \beta})} -
			y_{(\bs \beta)}}_{W_{\infty}} + \norm{\hat{u}_{(\hat{\bs \beta})} - u_{(\bs \beta)}}_{U} + \norm{\hat{\bs \beta} - \bs \beta}_{\calY} \right).
		\end{equation}
		\nin We also have
		\begin{equation*}
		u_{(\bs \beta)} = \mathbb{P}_{\mathcal{U}_{ad}} \left[\frac{1}{\alpha}\left(B^* p_{(\bs \beta)} + \beta_2 \right)\right] \in U \cap C(\bar I;\calU),
		\end{equation*}
		\nin and thus
		\begin{align*}
		\norm{\hat u_{(\hat{\bs{\beta}})}(t) - u_{(\bs{\beta})}(t)}_{\calU} &\leq \norm{\mathbb{P}_{\mathcal{U}_{ad}} \left[\frac{1}{\alpha} \left(B^*\hat{p}_{(\hat{\bs{\beta}})}(t) + \hat{\beta}_2(t) \right)\right] - \mathbb{P}_{\mathcal{U}_{ad}} \left[\frac{1}{\alpha}\left(B^* p_{(\bs \beta)}(t) + \beta_2(t) \right)\right]}_{\calU}, \\
		&\leq \frac{1}{\alpha} \left( \norm{B^*} \norm{\hat{p}_{(\hat{\bs{\beta}})}(t) - p_{(\bs \beta)}(t)}_{Y} + \norm{\hat{\beta}_2(t) - \beta_2(t)}_{\calU} \right).
		\end{align*}
		This yields
		\begin{equation}
		\norm{\hat u_{(\hat{\bs{\beta}})} - u_{(\bs{\beta})}}_{C(\bar I;\calU)} \leq C_{11} \left( \norm{\hat{p}_{(\hat{\bs{\beta}})} - p_{(\bs \beta)}}_{W_{\infty}} + \norm{\hat{\beta}_2 - \beta_2}_{C(\bar I;\calU)} \right),
		\end{equation}
		\nin and \eqref{est:p_dif} follows.	
	\end{proof}

\nin Combining Remark \ref{rmk-Vhat}, Step (iii) of the proof of Theorem \ref{thm:lip_con}, and \eqref{est:p_dif} there exists a constant $L$ such that
\begin{equation}
\norm{\hat{y}_{(\hat{\bs \beta})} - y_{(\bs \beta)}}_{W_{\infty}} + \norm{\hat{p}_{(\hat{\bs \beta})} - p_{(\bs \beta)}}_{W_{\infty}} + \norm{\hat{u}_{(\hat{\bs \beta})} - u_{(\bs \beta)}}_{U \cap C(\bar I;Y)} \leq L \norm{\hat{\bs \beta} - \bs \beta}_{\calY},
\end{equation}
for all $\ds \bs{\hat{\beta}}$ and $\ds \bs{\beta} \in \hat{V} \cap B_{\calY}(\varepsilon)$. Thus the verification of  \eqref{H0}--\eqref{H4} is concluded.
Here and in the following the $p_1$ coordinate of the adjoint state coincides with $p(0)$. Therefore it is not indicated.

 We now obtain the following corollary to Theorem \ref{thm:lip_con}.


	\begin{clr}\label{clr:lip_con}
		Let assumptions (\hyperref[assump]{A}) hold and let $(\bar{y}, \bar{u})$ be a local solution of (\hyperref[SLP]{$\calP$}) corresponding to an initial datum $\bar y_0 \in B_Y(\tilde \delta_2)$. Then there exist $\delta_3 > 0$, a neighborhood $\hat U = \hat U(\bar{y}, \bar{u}, p) \subset W_{\infty} \times (U \cap C(\bar{I};\calU)) \times W_{\infty}$, and a constant $\mu > 0$ such that for each $y_0 \in B_Y(\bar y_0, \delta_3)$ there exists a unique $(y(y_0), u(y_0), p(y_0)) \in \hat U$ satisfying the first order condition, and
		\begin{equation}\label{eq:clr_lc_2}
		\norm{ \left( y(\hat{y}_0), u(\hat{y}_0)), p(\hat{y}_0)) \right) - \left(y(\tilde{y}_0), u(\tilde{y}_0), p(\tilde{y}_0)) \right)}_{W_{\infty} \times (U \cap C(\bar{I};\calU)) \times W_{\infty}} \leq \mu \norm{\hat{y}_0 - \tilde{y}_0}_{Y},
		\end{equation}
for all $\hat{y}_0, \tilde{y}_0 \in B_Y(\bar y_0,\delta_3)$,  and $\left(y(y_0), u(y_0) \right)$ is a local solution of (\hyperref[SLP]{$\calP$}).
	\end{clr}

\nin Next we obtain one of the main results of this paper, the  Fr\'{e}chet differentiability of the local value function associated to  (\hyperref[SLP]{$\calP$}). By  referring to a local value function we pay attention to the fact that for some $y_0\in B_Y(\tilde \delta_2)$, problem  (\hyperref[SLP]{$\calP$})  may not admit a  unique solution. But since due to the second order optimality condition local solutions are locally unique  under small perturbations of $y_0$, there is a well-defined local value function. We  continue to use the notation for $\hat U$ and $B_Y(\bar y_0, \delta_3)$ of Corollary \ref{clr:lip_con}.

	\begin{thm}\label{thm-CD-r}
(Sensitivity of Cost)
Let assumptions (\hyperref[assump]{A}) hold and let $(\bar{y}, \bar{u})$ be a local solution of (\hyperref[SLP]{$\calP$}) corresponding to an initial datum $\bar y_0 \in B_Y(\tilde \delta_2)$. Then  for each $y_0 \in B_Y(\bar y_0, \delta_3)$
the local value function $\V$ associated to (\hyperref[SLP]{$\calP$}) is Fr\'{e}chet differentiable with derivative given by
		\begin{equation}
			\V'(y_0) = - p(0;y_0).
		\end{equation}
	\end{thm}
	\begin{proof}
Let  $\bar y_0 \in B_Y(\tilde \delta_2), y_0 \in B_Y(\bar y_0, \delta_3)$, and choose $\delta y_0$ sufficiently small so that  $y_0 + \delta y_0 \in B_Y(\bar y_0, \delta_3)$ as well. Following Corollary \ref{clr:lip_con} let $(\tilde{y}(y_0 + s(\delta y_0)), \tilde{u}(y_0 + s(\delta y_0)), \tilde{p}(y_0 + s(\delta y_0)))\in \hat U$ for $s\in [0,1]$ be solutions of the optimality system with  $(\tilde{y}(y_0 + s(\delta y_0)), \tilde{u}(y_0 + s(\delta y_0)))$  local solutions to (\hyperref[SLP]{$\calP$}).
 We obtain
		\begin{multline}
		\V(y_0 + s(\delta y_0)) - \V(y_0) = \left( \frac{1}{2} \norm{\tilde{y}}^2_{L^2(I,Y)} + \frac{\alpha}{2} \norm{\tilde{u}}^2_U\right) - \left( \frac{1}{2} \norm{y}^2_{L^2(I,Y)} + \frac{\alpha}{2} \norm{u}^2_U\right), \\= \ip{y}{\tilde{y} - y}_{L^2(I,Y)} + \alpha \ip{u}{\tilde{u} - u}_U + \frac{1}{2} \norm{\tilde{y} - y}^2_{L^2(I,Y)} + \frac{{\alpha}}{2} \norm{\tilde{u} - u}^2_{U}.
		\end{multline}
		\nin Observe the identity
		\begin{multline*}
		\ip{y}{\tilde{y} - y}_{L^2(I,Y)} + \alpha \ip{u}{ \tilde{u} - u}_U = {-}\ipp{p(0)}{s(\delta y_0)}_Y {-} \ip{(\tilde{y}_t - y_t) - \calA( \tilde{y} -  y) -  \calF'(y)(\tilde{y} - y)}{p} + \alpha \ip{u}{\tilde{u} - u}_U,\\
		= {-}\ipp{p(0)}{s(\delta y_0)}_Y {-} \ip{\calF(\tilde{y}) - \calF(y) - \calF'(y)(\tilde{y} - y)}{p}_{L^2(I;V^*), L^2(I;V)} + \ip{\alpha u - B^*p}{\tilde{u} - u}_U,
		\end{multline*}
		\nin where $p=p(y_0)$. Now we have for $\ds \V(y_0 + s(\delta y_0)) - \V(y_0)$,
		\begin{multline}
		\V(y_0 + {s}(\delta y_0)) - \V(y_0) = {-}\ipp{p(0)}{{s}(\delta y_0)}_Y + \ip{\calF(y) - \calF(\tilde{y}) + \calF'(y)(\tilde{y} - y)}{p}_{L^2(I,V^*),L^2(I,V)} \\ + \ip{\alpha u - B^*p}{\tilde{u} - u}_{U} + \frac{1}{2} \norm{\tilde{y} - y}^2_{L^2(I,Y)} + \frac{{\alpha}}{2} \norm{\tilde{u} - u}^2_{U}.
		\end{multline}
		\nin Since $p \in L^2(I;V)$, $\norm{\tilde{y} - y}_{W_{\infty}} = O(s)$, and by the continuous Fr\'{e}chet differentiability of $\calF'$ due to (\hyperref[assump]{A3}) we have
		\begin{equation}\label{F_o}
		\abs{\ip{\calF(\tilde{y}) - \calF(y) + \calF'(y)(\tilde{y} - y)}{p}_{L^2(I,V^*),L^2(I,V)}} = o(s).
		\end{equation}
		\nin Let $s_n \to 0$ be an arbitrary convergent sequence. By Corollary \ref{clr:lip_con} we have that
		\begin{equation*}
			\norm{\tilde{u}(y_0 + s_n(\delta y_0)) - u(y_0)}_U \leq \mu s_n(\delta y_0),
		\end{equation*}
		\nin for all $s_n$ sufficiently small. Hence there exists a subsequence, denoted by the same notation and some $\dot u$ such that
		\begin{equation*}
		s_n^{-1}\left(\tilde{u}(y_0 + s_n(\delta y_0)) - u(y_0)\right) \rightharpoonup \dot u \text{ weakly in } U.
		\end{equation*}
		\nin Using \eqref{adjP-3}, we have
		\begin{equation*}
		\lim_{n \rightarrow \infty} s_n^{-1} \ip{\alpha u - B^*p}{\tilde u - u}_U = \ip{\alpha u - B^*p}{\dot u}_U \geq 0.
		\end{equation*}
		\nin Analogously
		\begin{equation*}
		\lim_{n \rightarrow \infty} s_n^{-1} \ip{\alpha \tilde u - B^*p}{u - \tilde u}_U = \ip{\alpha u - B^*p}{\dot u}_U \leq 0.
		\end{equation*}
		and hence $\ip{\alpha u - B^*p}{\dot u}_U = 0$. Since the sequence $\{ s_n \}$ is arbitrary, we obtain
		\begin{equation}\label{u_dot_o}
		\ip{\alpha u - B^*p}{\tilde u - u}_U = o(s).
		\end{equation}
		\nin Corollary \ref{clr:lip_con} yields,
		\begin{equation}\label{yu_o}
		\norm{\tilde{y}(y_0 + {s_n}(\delta y_0)) - y(y_0)}^2_{L^2(I;Y)} + \alpha \norm{\tilde{u}(y_0 + {s_n}(\delta y_0)) - u(y_0)}^2_{L^2(I;Y)} = o(s_n).
		\end{equation}
		\nin Combining \eqref{F_o}, \eqref{u_dot_o}, and \eqref{yu_o} we obtain
		\begin{equation}
		\lim_{s \rightarrow 0^+} {s^{-1}}\left(\V(y_0 + s(\delta y_0)) - \V(y_0)\right) = {-}\ipp{p(0)}{(\delta y_0)}_Y.
		\end{equation}
This implies the Gateaux differentiability. Since $y_0 \to p(y_0)$ is continuous from $B_Y(\bar y_0;\delta_3)$ to $C(\bar I,Y)$ the mapping $y_0\to \mathcal{V}(y_0)$ is Fr\'{e}chet differentiable in $B_Y(\bar y_0;\delta_3)$.
	\end{proof}
		\begin{rmk}[Sensitivity w.r.t. other parameters]\label{rmk-para}
		We have developed a  technique to verify the continuous differentiability of the local value function $\V$ pertaining to a semilinear parabolic equation on infinite time horizon subject to control constraints with respect to  small  initial data $y_0 \in Y$. Thus the parameter $q$ in \eqref{Pq} is the initial condition $y_0$. The reason to focus on this case is due to feedback control. Without much additional effort the sensitivity analysis of the value function could be carried out with respect to other parameters as for instance additive noise on the right hand side of the state equation. The papers cited in the introduction, see e.g. \cite{GHH:2005}, \cite{GV:2006}, consider such situations for the finite horizon case.
 \end{rmk}	

	\section{Proof of Theorem \ref{thm-HJB}: Derivation of the HJB Equation.}\label{sec4a}

	Utilizing the results established so far we now verify  that the (global) value function $\V$ (i.e. the value function associated to global minima) is a solution to a Hamilton-Jacobi-Bellman equation. The initial conditions will be chosen from the neighborhood $Y_0$ of the origin in $Y$ so that the assertions of Theorem \ref{thm-CD-r} and Corollary \ref{clr:lip_con} are available. It will be convenient to recall the  dynamic programming principle for the infinite time horizon problem:  let $y_0$ be an initial condition for which  a solution to (\hyperref[SLP]{$\calP$}) exists. Then for all $\tau > 0$, we have
	\begin{equation}\label{d-prog-p}
	\V(y_0) = \inf_{u \in L^2(0,\tau;\mathcal{U}_{ad})} \int_{0}^{\tau} \ell (S(u,y_0;t),u(t)) dt + \V(S(u,y_0;\tau)),
	\end{equation}
	\nin where $\ds \ell(y,u) = \frac{1}{2} \norm{y}^2_Y  +  \frac{\alpha}{2} \norm{u}^2_{\calU}$, and $S(u,y_0;t)$ denotes the solution to \eqref{1.1a}, \eqref{1.1b} on $(0,\tau]$. \\
	
	\nin For convenience  we restate  Theorem \ref{thm-HJB}. Utilizing the notation that we have already established we can now slightly ease the assumption on the regularity of $\calF (\bar y)$.
	
	\begin{thm}\label{thm-HJB-r}
Let assumptions (\hyperref[assump]{A}) hold and let $(\bar{y}, \bar{u})$ be a global solution of (\hyperref[SLP]{$\calP$}) corresponding to an initial datum $\bar y_0 \in B_Y(\tilde \delta_2)$.
Let $Y_0$ denote the subset of initial conditions in  $B_Y(\bar y_0, \delta_3)$ which allow global solutions in $\hat U$, and   assume that for each $y_0\in \calD(\calA) \cap Y_0$ there exists $T_{y_0}>0$ such that $\calF (\bar y)\in C([0,T_{y_0});Y)$. Then the following Hamilton-Jacobi-Bellman equation holds at $y_0$:
		\begin{equation} \label{eq:5.3}
		\V'(y)(\calA y + \calF(y)) + \frac{1}{2} \norm{y}^2_Y + \frac{\alpha}{2} \norm{ \mathbb{P}_{\mathcal{U}_{ad}} \left(\frac{1}{\alpha}B^*\V'(y) \right)}^2_Y + \left\langle B^* \V'(y),\mathbb{P}_{\mathcal{U}_{ad}} \left(\frac{1}{\alpha}B^*\V'(y) \right)\right\rangle_Y = 0.
		\end{equation}
If for the optimal trajectory  $\bar y(t) \in B_Y(\bar y_0, \delta_3)\cap \calD(\calA) $ for a.a. $t\in (0,\infty$) and $T_{y_0}=\infty$, then 	\eqref{eq:5.3} holds at a.a. $t\in (0,\infty$) and
		\begin{equation}\label{eq:5.4}
		{\bar u}(t) =  \mathbb{P}_{\mathcal{U}_{ad}} \left(\frac{1}{\alpha}  B^*\V'({\bar y}(t)) \right).
		\end{equation}
	\end{thm}
	
	\begin{proof}
		\nin The proof is similar to that of \cite[Proposition 10]{BKP:2018}. For the sake of completeness and since it also requires some changes we provide it here. Choose and fix some  $y_0 \in \calD(\calA) \cap Y_0$. Then the existence of a (globally) optimal pair $(\hat{y}, \hat{u}) \in W_{\infty} \times \Uad$ to (\hyperref[SLP]{$\calP$}) and of an associated adjoint state $\hat p\in W_\infty$ with $(\hat y,\hat u, \hat p) \in \hat U$  are guaranteed, see Corollary \ref{clr:lip_con}. In particular we have that $\ds \hat{u}(t) = \mathbb{P}_{\mathcal{U}_{ad}} \left( \frac{1}{\alpha}B^*{\hat p}(t) \right)$, and since $\ds \hat p \in C([0,\infty);Y)$ we have that $\ds \hat u\in C([0,\infty);Y)$. Let $u_0$ denote the limit of $\hat{u}$ as time $t$ tends to $0$. Since $\hat{y} \in C([0,\infty); Y)$ and since $B_Y(y_0,\delta_3) $ is open there exists $\tau_{y_0} > 0$  such that $\hat{y}(t) \in B_Y(y_0,\delta_3) $, for all $t \in [0,\tau_{y_0})$.\\
		
		\nin \underline{Step 1:}	
		Let us first prove that
		\begin{equation}\label{hjb.1}
		\V'(y_0) \big(\calA y_0 + \calF(y_0) + B u_0 \big) + \ell(y_0, u_0) = 0.
		\end{equation}
		For this purpose we invoke the dynamic programing principle:  We have
		\begin{equation}\label{v-lim}
		\frac{1}{\tau} \int_{0}^{\tau} \ell(\hat{y}(s), \hat{u}(s))ds + \frac{1}{\tau} \big( \V(\hat{y}(\tau)) - \V(y_0) \big) = 0,
		\end{equation}
		where we choose  $\tau \in (0,\min(T_{y_0}, \tau_{y_0}))$ . By continuity of $\hat{y}$ and $\hat{u}$ at time $0$, the first term  converges to $\ds \ell(y_0, u_0)$ as $\tau \to 0$. To take $\tau \to 0$ in the second term we first consider
		\begin{equation}\label{eq:aux5}
		\frac{1}{\tau}\big( \hat{y}(\tau) - y_0 \big) = \frac{1}{\tau}\big(e^{\calA \tau}y_0 - y_0 \big) + \frac{1}{\tau} \int_{0}^{\tau} e^{\calA(\tau - s)} \big[ \calF(\hat{y}(s)) + B \hat{u}(s) \big] ds.
		\end{equation}
		Using the facts that $y_0\in \calD(\calA)$, that  the terms in square brackets are continuous with values in $Y$, and that $\calA$ generates a strongly continuous semigroup on $Y$, we can pass to the limit in \eqref{eq:aux5} to obtain that
		\begin{equation}\label{y-lim}
		\lim_{\tau \to 0^+}\frac{1}{\tau}\big( \hat{y}(\tau) - y_0 \big)  = \calA y_0 + \calF(y_0) + B u_0 \text{ in } Y.
		\end{equation}
		Now we return to the second term in \eqref{v-lim} which we express as
		\begin{equation}\label{eq:aux6}
		\frac{1}{\tau} \big( \V(\hat{y}(\tau)) - \V(y_0) \big)=  \int_{0}^{1} \V' \big(y_0 + s ( \hat{y}(\tau) - y_0 ) \big)ds \; \frac{1}{\tau}( \hat{y}(\tau) - y_0 ).
		\end{equation}
		Using \eqref{y-lim} and since $y\to  \V' (y)$ is continuously differentiable at  $y_0$,  we can pass to the limit in \eqref{eq:aux6} to obtain
		\begin{equation}
		\lim_{\tau \to 0^+}\frac{1}{\tau} \big( \V(\hat{y}(\tau)) - \V(y_0) \big)= \V'(y_0) \big(\calA y_0 + \calF(y_0) + B u_0 \big).
		\end{equation}
		\nin Now we can pass to the limit in \eqref{v-lim} and obtain \eqref{hjb.1}.\\
				
		\nin \underline{Step 2:} For $u \in \mathcal{U}_{ad}$  we define $\tilde{u}\in U_{ad}$ by,
		\begin{equation*}
		\tilde{u}(\tau,x) =
		\begin{cases}
		u \quad \text{for } \tau \in (0,1),\\
		0 \quad \text{ for } \tau \in [1, \infty),
		\end{cases}
		\end{equation*}
		\nin and define $\tilde{y} = S(y_0,\tilde{u})$ as the solution to \eqref{1.1a}, \eqref{1.1b}. Then $\tilde y(t) \in B_Y(\bar y_0,\delta_3)$, for all $t$ sufficiently small, and  by \eqref{d-prog-p} we have,
		\begin{equation*}
		\frac{1}{\tau} \int_{0}^{\tau} \ell(\tilde{y}(s), u(s))ds + \frac{1}{\tau} \big( \V(\tilde{y}(\tau)) - \V(y_0) \big) \geq 0,
		\end{equation*}
		for all $\tau$ sufficiently small.
		We pass to the limit $\tau \to 0^+$ with the same arguments as in Step 1 and obtain  \begin{equation}\label{hjb.2}
		\V'(y_0) \big(\calA y_0 + \calF(y_0) + B u \big) + \ell(y_0, u) \ge 0.
		\end{equation}
		\nin This inequality becomes an equality if $u = u_0$, and thus the quadratic function on the left had side of \eqref{hjb.2} reaches its minimum $0$ at $u = u_0$. This implies that
		$u_0 =  \mathbb{P}_{\mathcal{U}_{ad}} \left(\frac{1}{\alpha} B^* \V'(y_0) \right).$
		Inserting this expression into  \eqref{hjb.1} we obtain
		\begin{equation}
		\V'(y_0)(\calA y_0 + \calF(y_0)) + \frac{1}{2} \norm{y_0}^2_Y + \frac{\alpha}{2}\norm{ \mathbb{P}_{\mathcal{U}_{ad}}\left(\frac{1}{\alpha} B^* \V'(y_0) \right)}^2_Y + \left\langle B^* \V'(y_0),\mathbb{P}_{\mathcal{U}_{ad}}\left(\frac{1}{\alpha} B^* \V'(y_0) \right) \right\rangle_Y = 0.
		\end{equation}
Under the additional assumptions on the trajectory, \eqref{eq:5.4} follows.
		
	\end{proof}

	\section{Some Applications}\label{sec6}
	\nin  In this section we discuss the applicability of the framework in two specific cases. It should be noted that even for linear state equations, the sensitivity result for the constraint infinite horizon optimal control  problem may be new.
	
	
	
	\subsection{Fisher's Equation}
	\nin We consider the optimal stabilization problem  for the Fisher equation in an open connected bounded domain $\Omega$ in $\BR^d, \ d \in \{1,2,3,4\}$, with Lipschitzian boundary $\Gamma = \partial \Omega$:\\
	\begin{subequations}
		\nin
		\begin{align}
		(\calP_{Fis}) \qquad \V(y_0) = \min_{\begin{matrix}
			(y,u) \in W_\infty \times U_{ad} \end{matrix}} \ \frac{1}{2} \int_{0}^{\infty} \norm{y}^2_Y dt + \frac{\al}{2} \int_{0}^{\infty} \norm{u}^2_{\calU} dt,
		\end{align}	
		\nin subject to
		\begin{empheq}[left=\empheqlbrace]{align}
		y_t &= \Delta y + y(1-y) + Bu &\text{ in } Q = (0, \infty) \times \Omega,\\
		y &= 0 &\text{ on } \Sigma = (0, \infty) \times \Gamma,\\
		y(0) &= y_0 &\text{ in } \Omega,
		\end{empheq}
	\end{subequations}
\nin where $\calU$ and $U_{ad}$ are as in Section \ref{sec3.1}, $B \in \calL({\calU, Y})$, with $Y=L^2(\Omega)$ and $V = H^1_0(\Omega)$. To further cast this problem in the framework of Section \ref{sec3}, we define the operator
\begin{equation*}
	\calA y = (\Delta + \mb{I}) y, \quad \text{and} \quad y|_{\Gamma} = 0, \quad \calD(\calA) = H^2(\Omega) \cap V.
	\end{equation*}
Clearly  $\calA$ has an extension as operator $\calA \in \calL(V,V^*)$. Moreover it generates an analytic semigroup on $Y$. Thus (\hyperref[assump]{A1}) holds. For $\calU = Y$ and $B = \mb{I}$, condition (\hyperref[assump]{A2}) is trivially satisfied. Feedback stabilization by finite dimensional controllers was analyzed in \cite{RT:1975}, for example.\\

\nin It can readily be checked that  the nonlinearity  $\calF(y) = -y^2$  is twice continuously differentiable as mapping $\calF: W_{\infty} \to L^2(I;V^*)$. The first and second derivatives of $\calF$ are given by,
\begin{align*}
	\calF'(y)v_1 = 2yv_1, \quad
	\calF''(y)(v_1,v_2) = 2(v_1,v_2), \quad \text{ for }  y,v_1,v_2 \in W_{\infty}.
\end{align*}
\nin Since the second derivative is independent of $y$,  its boundedness  is automatic. For the sake of illustration we verify  the boundedness of the bilinear form of the second derivative on $W_\infty \times W_\infty$.
For this purpose, for arbitrary $y\in W_\infty, v_1, v_2\in W_\infty, \phi\in L^2(I;V)$ we estimate
 \begin{equation}\label{eq:6.2}
\begin{aligned}
\int_{0}^{\infty} \ip{\calF''({y})(v_1, v_2)}{\phi}_{V^*,V} dt &\le  2 \int_{0}^{\infty} \int_{\Omega} v_1 v_2 \phi \ dx dt \le \int^\infty_0 \norm{v_1}_{L^2(\Omega)} \norm{v_2}_{L^4(\Omega)} \norm{\phi}_{L^4(\Omega)}\,dt,\\
&\le C_1 \norm{v_1}_{W_{\infty}} \int^\infty_0  \norm{v_2}_{V} \norm{\phi}_{V}\,dt
\le C_2 \norm{v_1}_{W_{\infty}}  \norm{v_2}_{L^2(I;V)} \norm{\phi}_{L^2(I;V)},\\
&\le C_3 \norm{v_1}_{W_{\infty}}   \norm{v_2}_{W_{\infty}} \norm{\phi}_{L^2(I;V)},
\end{aligned}
 \end{equation}

\nin where $C_i$ are embedding constants, independent of  $y\in W_\infty, v\in W_\infty, \phi\in L^2(I;V)$. We use that $V$ embeds continuously into $L^4(\Omega)$ in dimension up to 4.
 This implies that $\norm{\calF''({y})(v_1, v_2)}_{L^2(I;V^*)} \le  C_3 \norm{v_1}_{W_{\infty}} \norm{v_2}_{W_{\infty}}$.  Finally we have $\calF(0) = \calF'(0) = 0$  and thus (\hyperref[assump]{A3}) and \eqref{eq:fprime} are satisfied.\\

\nin Turning to (\hyperref[assump]{A4}) we show that $\calF(y):W(0,T) \rightarrow L^1(0,T;V^*)$ is continuous for every $T > 0$. We consider the sequence $ \ds y_n \rightharpoonup \hat{y}$ in $W_{\infty}$ and let $z \in L^{\infty}(0,T; V)$ be given.
Then we estimate
	\begin{align*}
\int_{0}^T \ip{\calF(y_n) - \calF(\hat{y})}{z}_{V^*,V} dt & =	\int_{0}^T \ip{y_n^2 - \hat{y}^2}{z}_{V^*,V} = \int_{0}^T \int_{\Omega}(y_n - \hat{y})(y_n + \hat{y})z \ dx dt,\\
	&\leq C_4 \int_{0}^T \norm{y_n - \hat{y}}_{Y}\norm{y_n + \hat{y}}_{L^4(\Omega)}\norm{z}_{L^4(\Omega)} \ dt,\\
	&\leq C_4 \norm{y_n - \hat{y}}_{L^2(0,T;Y)} \left[  \norm{y_n}_{L^2(0,T;V)} + \norm{\hat{y}}_{L^2(0,T;V)}\right] \norm{z}_{L^{\infty}(0,T;V)}.
	\end{align*}
	Since $V$ is compactly embedded in $Y$, we obtain by the Aubin Lions lemma that $\ds \norm{y_n - \hat{y}}_{L^2(0,T;Y)} \to 0$ for  $n \to \infty$. This implies
	\begin{equation*}
	\int_{0}^T \ip{\calF(y_n) - \calF(\hat{y})}{z}_{V^*,V} dt \xrightarrow[n \rightarrow \infty]{} 0,
	\end{equation*}
and (\hyperref[assump]{A4}) follows.
It is simple to check that $\ds \calF'(\bar y) = 2\bar y \in \calL(L^2(I;V),L^2(I;V^*))$ and thus (\hyperref[assump]{A5}) holds as well.\\

\nin We turn to the assumption  $\calF (\bar y)\in C([0,T_{y_0});Y)$, for $y_0 \in \calD(\calA)$ and some $T_{y_0}$,  arising in Theorem \ref{thm-HJB} for $y_0 \in \calD(\calA)$. Utilizing the fact that $V$ embeds continuously into $L^4(\Omega)$ in dimension $d\le 4$ and $\bar y \in L^2(I;V)$, we have $\calF(\bar y) \in L^2(I;Y)$. Hence parabolic regularity theory implies that $\bar y \in C([0,\infty);V)$ for $y_0 \in V$, and $\calF (\bar y)\in C([0,\infty);Y)$ follows.

\begin{rmk}
The specificity of this example rests in the fact that the second derivative is independent of the point were it is taken. Other nontrivial cases of analogous structure are reaction diffusion systems with bilinear coupling, see \cite{Gri:2004} where the finite horizon case was treated. Even the case of the Navier Stokes equations falls in this category. Sensitivity for the infinite horizon problems was treated by independent techniques  in \cite {BKP:2019}.
\end{rmk}

\subsection{Nonlinearities induced by functions with globally Lipschitz continuous second derivative.}
	
	Consider the system (\hyperref[Pq]{$\calP$}) with $\calA$ associated to a strongly elliptic second order operator with domain $H^2(\Omega) \cap H^1_0(\Omega)$, so that (\hyperref[assump]{A1})-(\hyperref[assump]{A2}) are satisfied. Let $\calF: W_{\infty} \to L(I;V^*)$ be the Nemytskii operator associated to a mapping $\mathfrak f: \BR \to \BR$ which is assumed to be $C^2(\BR)$ with first and second derivatives globally Lipschitz continuous, and second derivative globally bounded. The regularity assumption   $\calF (\bar y)\in C([0,T_{y_0});Y)$ for $y_0 \in V=H^1_0(\Omega)$  is satisfied by parabolic regularity theory.  We discuss assumption (\hyperref[assump]{A3})-(\hyperref[assump]{A5}) for such an  $\calF$, and show that they are satisfied for dimensions $d\in \{1,2\}$. For the finite horizon problem it will turn out that $d=3$ is also admissible. By direct calculation it can be checked that  $\calF$ is continuously Fr\'{e}chet differentiable for $d \in \{ 1,2,3 \}$. We leave this part to the reader and immediately turn to the second derivative.\\

\nin We proceed by considering the general dimension $d$ to highlight, how the restrictions on the dimension arise. Thus let $d \in \mathbb{N} $ with $d>1$. The case $d=1$ can be treated with minor modifications from those in the following steps.

\subsubsection{Second derivative of $\calF(y)$.}\label{sdF}
For $y, h_1, h_2 \in W_\infty $  the relevant  expression is given by
	\begin{multline*}
	\norm{\calF'(y + h_2)h_1 - \calF'(y)h_1 - \calF''(y)(h_1,h_2)}_{L^2(I;V^*)}\\
	\quad = \sup_{\norm{\varphi}_{L^2(I;V)} \leq 1}\ip{\calF'(y + h_2)h_1 - \calF'(y)h_1 - \calF''(y)(h_1,h_2)}{\varphi}_{L^2(I;V^*),L^2(I;V)},\\
	\quad= \sup_{\norm{\varphi}_{L^2(I;V)} \leq 1} \int_{0}^{\infty} \int_{\Omega} (\mathfrak{f}'(y(t,x) + h_2(t,x)) - \mathfrak{f}'(y(t,x)) - \mathfrak{f}''(y(t,x))h_2(t,x)) h_1(t,x)\varphi(t,x) dxdt,\\
	\quad = \sup_{\norm{\varphi}_{L^2(I;V)} \leq 1} \int_{0}^{\infty} \int_{\Omega} g(t,x) \ h_2(t,x)h_1(t,x)\varphi(t,x) dxdt,
\end{multline*}
\nin where  $\ds g(t,x) = \int_0^1 (\mathfrak{f}''(y(t,x) + sh_2(t,x)) - \mathfrak{f}''(y(t,x)))ds$.
\nin Note that $g$ is bounded on $I \times \Omega$ and $g \in W_{\infty}$. Here we use that $\mathfrak{f}''$ is globally Lipschitz continuous and that $h_1 \in W_\infty$.
Henceforth we let $r\in (1, \frac{2d}{d-2}]$ so that $W^{1,2}(\Omega)\subset L^r(\Omega)$ continuously.
Let $r'$ denote the conjugate of $r$ so that $ r'\in [\frac{2d}{d+2} ,\infty)$ for $d>2$ and $r'\in (1,\infty)$ for $d=2$.
We further choose $\rho > 1, \sigma > 2$ such that
$\frac{1}{\rho} + \frac{2}{\sigma} = 1$. Then we estimate
\begin{align*}
\abs{\int_{0}^{\infty} \int_{\Omega} gh_1h_2\varphi \ dxdt} &\leq
\int^\infty_0 \left( \int_\Omega \abs{g h_1 h_2}^{r'} dx  \right)^{\rfrac{1}{r'}} \norm{\varphi(t)}_{L^r(\Omega)} \, dt,\\
&\le \left(\int^\infty_0 \left( \int_\Omega \abs{g h_1 h_2}^{r'} dx \right)^{\rfrac{2}{r'}}dt \right)^{\rfrac{1}{2}} \left(\int^{\infty}_0\norm{\varphi}^2_{L^r(\Omega)}dt \right)^{\rfrac{1}{2}}.
\end{align*}
This further implies that
\begin{equation}\label{est_ghhp}
		\abs{\int_{0}^{\infty} \int_{\Omega} gh_1h_2\varphi \ dxdt} \leq C_0 \left[\int_{0}^{\infty} \norm{g(t)}^2_{L^{r'\rho}(\Omega)} \norm{h_1(t)}^2_{L^{r'\sigma}(\Omega)} \norm{h_2(t)}^2_{L^{r'\sigma}(\Omega)}  dt\right]^{\rfrac{1}{2}} \norm{\varphi}_{L^2(I;V)}. \longleftarrow (a)
	\end{equation}
Here and below  $C_i, i=0,1,2, \dots$  denote constant which are independent of $y,\varphi, h_1, h_2$. We next recall Gagliardo's inequality \cite[p 173]{BF:2013}:
	\begin{equation*}
	\norm{u}_{L^q(\Omega)} \leq \norm{u}^{1 - \rfrac{d}{q}-\rfrac{d}{2}}_{L^2(\Omega)}\norm{u}^{\rfrac{d}{2}+\rfrac{2}{q}}_{W^{1,2}(\Omega)}, \text{ for all } q > 2, \text{ and } u \in W^{1,2}(\Omega) \equiv V,
	\end{equation*}
where $q\in[2,2^*]$ and

\begin{equation*}
q^* \begin{cases}
       \in [2,\infty] \text{ for } d = 1, \\
       \in [2,\infty) \text{ for } d = 2, \\
       \in [2, \frac{2d}{d-2}],  \text{ for } d > 2.
    \end{cases}
\end{equation*}

	\nin In the above estimate we take, $q = r'\sigma$. We obtain
	\begin{equation*}
		1 + \frac{d}{q} - \frac{d}{2} = \frac{(2-d)r'\sigma + 2d}{2r'\sigma}, \ \frac{d}{2} - \frac{d}{q} = \frac{d(r'\sigma - 2)}{2r'\sigma}, \text{ also } r'\sigma > 2, \ (2-d)r'\sigma + 2d > 0.
	\end{equation*}
	\nin We estimate \eqref{est_ghhp}, (and check the conditions on the ranges of the parameters below)
	\begin{align}
		&\sup_{\norm{\varphi}_{L^2(I;V)} \leq 1} (a) \\
 &= C_1 \left( \int_{0}^{\infty} \norm{g(t)}^2_{L^{r'\rho}(\Omega)} \left( \norm{h_1(t)}_{L^2(\Omega)} \norm{h_2(t)}_{L^2(\Omega)} \right)^{\frac{(2-d)r'\sigma + 2d}{r'\sigma}} \left( \norm{h_1(t)}_{V} \norm{h_2(t)}_{V} \right)^{\frac{d(r'\sigma - 2)}{r'\sigma}} \right)^{\rfrac{1}{2}}dt, \nonumber\\
		&\leq C_2 \left( \norm{h_1}_{W_{\infty}} \norm{h_2}_{W_{\infty}} \right)^{\frac{(2-d)r'\sigma + 2d}{2r'\sigma}} \left( \int_{0}^{\infty} \norm{g(t)}^2_{L^{r'\rho}(\Omega)}  \left( \norm{h_1(t)}_{V} \norm{h_2(t)}_{V} \right)^{\frac{d(r'\sigma - 2)}{r'\sigma}} \right)^{\rfrac{1}{2}}dt. \longleftarrow (b)\nonumber
	\end{align}
 	\nin We set $\frac{d(r'\sigma - 2)}{r'\sigma} = \frac{2}{3}$. This yields,
	\begin{align}
	(b) & \le C_3 \left( \norm{h_1}_{W_{\infty}} \norm{h_2}_{W_{\infty}} \right)^{\rfrac{2}{3}} \left( \norm{h_1}_{W_{\infty}} \norm{h_2}_{W_{\infty}} \right)^{\rfrac{1}{3}} \left(\int_{0}^{\infty} \norm{g(t)}^6_{L^{r'\rho}(\Omega)}\right)^{\rfrac{1}{6}} dt, \nonumber\\
		&= C_4 \norm{h_1}_{W_{\infty}} \norm{h_2}_{W_{\infty}} \left( \int_{0}^{\infty} \left(\int_{\Omega} \abs{g(t)}^{r'\rho} dx \right)^{\rfrac{6}{r'\rho}} dt \right)^{\rfrac{1}{6}} \label{eq_b}.
	\end{align}
	\nin Now we check the conditions  on the parameter $r, \sigma, d$ , and $r', \sigma$. Since $\frac{d(r'\sigma - 2)}{r'\sigma} = \frac{2}{3}$, together with the conditions on $r'$ and $\sigma$ these parameters need to satisfy
\begin{equation}\label{eq:6.4}
r'\sigma = \frac{6d}{3d-2}, \quad  r'\in \left[\frac{2d}{d+2}, \infty \right), \quad \sigma \in (2,\infty)\quad r'\sigma\in[2,2^*] ,
\end{equation}
and $r'>1$ if $d=2$. The last condition above holds without restricting the dimension $d$. From the first three  relations we infer that necessarily $\frac{6d}{3d-2} =r'\sigma >\frac{4d}{d+2}$ which is only possible for $d\le3$.\\

\nin {\em Let us focus on $d=2$.} Then the choice of parameters $r=6,r'=\rfrac{6}{5}, \sigma=\rfrac{5}{2}, \rho=5$ satisfies all the above requirements and it is convenient to further estimate \eqref{eq_b}.
In fact we obtain
	\begin{multline*}
	\norm{\calF'(y + h_2)h_1 - \calF'(y)h_1 - \calF''(y)(h_1,h_2)}_{L^2(I;V^*)} \\
	\le C_5 \norm{h_1}_{W_{\infty}} \norm{h_2}_{W_{\infty}} \left(\int_{0}^{\infty} \left(\int_{\Omega} \abs{g(t,x)}^{6}dx \right) dt \right)^{\rfrac{1}{6}},\\
 \le C_6 \norm{h_1}_{W_{\infty}} \norm{h_2}_{W_{\infty}} \left(\int_{0}^{\infty} \left(\int_{\Omega} \abs{g(t,x)}^{2}dx \right) dt \right)^{\rfrac{1}{6}},
\end{multline*}
\nin for all $y, h_1, h_2 \in W_\infty $. Here we use the boundedness of $g$. By Lebesgue's bounded convergence theorem the last factor converges to $0$ for $\norm{h_2}_{W_\infty} \to 0$ and hence the fact that $\calF$ is twice differentiable is verified. The continuity of the second derivative follows with the above estimates and again by the Lebesgue theorem.\\

\nin {\em Next we consider $d=3$.} In this case an analogous procedure is not possible, since the relations \eqref{eq:6.4} and $r'\rho \le 6$ cannot be fulfilled simultaneously. In fact, $r'\sigma = \rfrac{18}{7}, r'\ge \rfrac{6}{5}$, and thus necessarily $\sigma\in (2, \rfrac{15}{7}]$. The condition $r'\rho \le 6$ is equivalent to $12 \le \sigma(6-r')=\rfrac{18}{7r'}(6-r')$,  which in turn is equivalent to $17r' \le 18$, which contradicts $r' \ge \rfrac{6}{5}$.\\

\nin Thus we fix parameters $r$ and $\sigma$ such that \eqref{eq:6.4} are satisfied for $d=3$, as for instance $r'=\rfrac{6}{5}, \sigma= \rfrac{15}{7}$, which implies that $\rho=15$ and $r'\rho=18$.
 Then for the {\em finite horizon problem} we can estimate by H\"older's inequality with $\eta= \rfrac{r'\rho}{6}$:
\begin{align*}
		(b) &\le C_7 \norm{h_1}_{W_{\infty}} \norm{h_2}_{W_{\infty}} \left( \int_{0}^{T} \left(\int_{\Omega} \abs{g(t,x)}^2 dx \right)^{\rfrac{6}{r'\rho}}dt \right)^{\rfrac{1}{6}} ,\nonumber\\
		&\leq C_8 \norm{h_1}_{W_{\infty}} \norm{h_2}_{W_{\infty}} \left( \int_{0}^{T} \left(\int_{\Omega} \abs{g(t,x)}^2 dx \right) dt\right)^{\rfrac{1}{6\eta}} T^{\rfrac{1}{6\eta'}}.
	\end{align*}
From here we can proceed as in the case $d=2$ to assert the continuous second Fr\'echet differentiability of $\calF$ in $d=3$ for the finite horizon case.

\subsubsection{Assumptions (\hyperref[assump]{A4}) and (\hyperref[assump]{A5}).}\label{s622}
	\nin In order to verify (\hyperref[assump]{A4}), we show that $\calF(y):W(0,T) \to L^1(0,T;V^*)$ is continuous for every $T > 0$. We consider the sequence $y_n \rightharpoonup \hat{y}$ in $W(0,T)$ and let $z \in L^{\infty}(0,T;V)$ be given. Then we estimate
	\begin{equation*}
		\int_{0}^T \ip{\calF(y_n) - \calF(\hat{y})}{z}_{V,V^*} dt = \int_{0}^T \int_{\Omega} (\mathfrak{f}(y_n) - \mathfrak{f}(\hat{y}))z \ dxdt \leq C \norm{y_n - \hat{y}}_{L^2(0,T;Y)}\norm{z}_{L^2(0,T;Y)}.
	\end{equation*}
	\nin Then by the compactness of $V$ in $Y$, we obtain (\hyperref[assump]{A4}).\\
	
	\nin Now we verify (\hyperref[assump]{A5}). We recall Remark \ref{rmk-fprime}, and proceed as in \eqref{eq:6.2}  for  $y \in W_{\infty}, \varphi \in L^2(I;V)$,
	\begin{eqnarray}
		\norm{\calF'(y)^*p}_{L^2(I;V^*)} = \norm{(\calF'(y)^* - \calF'(0)^*)p}_{L^2(I;V^*)} = \sup_{\norm{\varphi}_{L^2(I;V)} \leq 1} \int_{0}^{\infty}\int_{\Omega} \ip{(\calF'(y)^* - \calF'(0)^*)p}{\varphi}_{V^*,V}, \nonumber\\
		= \sup_{\norm{\varphi}_{L^2(I;V)} \leq 1} \int_{0}^{\infty}\int_{\Omega} (\mathfrak{f}'(y) - \mathfrak{f}'(0))p\varphi \ dxdt	\leq C\norm{y}_{W_{\infty}}\norm{p}_{L^2(I;V)}. \nonumber
	\end{eqnarray}
	\nin This shows $\calF'(y)^*$ satisfies (\hyperref[assump]{A5}).
		
	\subsection{Cubic nonlinearity $y^3$ in one dimension ($\Omega \subset \BR$).}
	\nin We can also consider the optimal stabilization problem with cubic nonlinearity, i.e. $\calF(y) = y^3$ in one dimension. This is a special monotone case of the Schl\"{o}gl model of theoretical chemistry.
	\begin{subequations}
		\nin
		\begin{align}
		(\calP_{Sch}) \qquad \V(y_0) = \min_{\begin{matrix}
			(y,u) \in W_\infty \times U_{ad} \end{matrix}} \ \frac{1}{2} \int_{0}^{\infty} \norm{y}^2_Y dt + \frac{\al}{2} \int_{0}^{\infty} \norm{u}^2_{\calU} dt,
		\end{align}	
		\nin subject to
		\begin{empheq}[left=\empheqlbrace]{align}
		y_t &= \Delta y + y^3 + Bu \quad &\text{ in } Q = (0, \infty) \times \Omega,\\
		y &= 0 \quad &\text{ on } \Sigma = (0, \infty) \times \Gamma,\\
		y(0) &= y_0 \quad &\text{ in } \Omega.
		\end{empheq}
	\end{subequations}
	\nin In this model, one can easily verify assumption (\hyperref[assump]{A1}) is satisfied by taking $\calA y= \Delta y, \ y|_{\Gamma} = 0,$ and $\calD(\calA) = H^2(\Omega) \cap V$. Clearly $\calA$ can be extended to $\calA \in \calL(V,V^*)$. Moreover $\calA$ generates an analytic semigroup on $Y$ which is uniformly stable. Assumption (\hyperref[assump]{A2}) is satisified under the same argumentation as in Fisher's equation. Differentiability assumption (\hyperref[assump]{A3}), and continuity assumption (\hyperref[assump]{A4}) are satisfied along similar computations as in subsections \ref{sdF}, \ref{s622}. For   (\hyperref[assump]{A5}) we require that $y_0\in V$.  Indeed in this case for $\bar y \in W_\infty$ by Gagliardo's inequality
\begin{equation*}
\int_0^\infty \int_\Omega |\bar y^3|^2 dx dt = \int_0^\infty \norm{\bar y}^6_{L^6(\Omega)} dt \le  \int^\infty_0\norm{\bar y}^4_{L^2(\Omega)} \norm{\bar y}^2_{V} dt \le C \norm{\bar y}^4_{W_\infty}\int^\infty_0   \norm{\bar y}^2_{V} dt \le C \norm{\bar y}^6_{W_\infty}.
\end{equation*}
Thus $\bar y^3 \in L^2(I;Y)$ and parabolic regularity theory implies that $\bar y\in C(I;V)$ if $y_0\in V$.
We estimate for $h,\varphi \in L^2(I;V)$, suppressing the arguments $(t,x)$,
\begin{align*}
\abs{\int_0^\infty \int_{\Omega} \calF'(\bar y) h \varphi \, dx dt} \le \abs{\int_0^\infty \int_{\Omega} \bar y^2 h \varphi \, dx dt} &\le \int_0^\infty \norm {\bar y}^2_{L^4(\Omega)} \norm {h}_{L^4(\Omega)} \norm {\varphi}_{L^4(\Omega)} \, dt, \\
&\le C \norm{\bar y}^2_{C(I;V)} \norm{ h}_{L^2(I;V)} \norm{\varphi}_{L^2(I;V)}
\end{align*}
which implies (\hyperref[assump]{A5}). Moreover we have $\calF (\bar y)\in C([0,T_{y_0});Y)$, since $V \subset C(\bar \Omega)$ in dimension 1, and thus the extra regularity demanded in Theorem \ref{thm-HJB} is satisfied.

%
%

	\medskip
	Received xxxx 20xx; revised xxxx 20xx.
	\medskip
	
\end{document}